\pdfoutput=1

\documentclass{siamart0216}

\usepackage{fullpage}
\usepackage{hyperref}
\usepackage{amsmath,amssymb,amsfonts,mathrsfs}%,amsthm}
\usepackage[titletoc,toc,title]{appendix}

\usepackage{array} 
\usepackage[utf8]{inputenc}
\usepackage{listings}
\usepackage{mathtools}
\usepackage{pdfpages}
\usepackage[textsize=footnotesize,color=green]{todonotes}
\usepackage{bm}
\usepackage[normalem]{ulem}
\usepackage{hhline}

\usepackage{algorithm}
\usepackage[noend]{algpseudocode}
\usepackage{algorithmicx}
\algblock{ParFor}{EndParFor}
\algnewcommand\algorithmicparfor{\textbf{parfor}}
\algnewcommand\algorithmicpardo{\textbf{do}}
\algnewcommand\algorithmicendparfor{\textbf{end\ parfor}}
\algrenewtext{ParFor}[1]{\algorithmicparfor\ #1\ \algorithmicpardo}
\algrenewtext{EndParFor}{\algorithmicendparfor}
\usepackage{graphicx}
\usepackage{subfig}
\usepackage{color}

\usepackage{pgfplots}
\usepackage{pgfplotstable}
\definecolor{markercolor}{RGB}{124.9, 255, 160.65}
\pgfplotsset{width=10cm,compat=1.3}
\pgfplotsset{
tick label style={font=\small},
label style={font=\small},
legend style={font=\small}
}

\usetikzlibrary{calc}

\newcommand{\logLogSlopeTriangle}[5]
{
    \pgfplotsextra
    {
        \pgfkeysgetvalue{/pgfplots/xmin}{\xmin}
        \pgfkeysgetvalue{/pgfplots/xmax}{\xmax}
        \pgfkeysgetvalue{/pgfplots/ymin}{\ymin}
        \pgfkeysgetvalue{/pgfplots/ymax}{\ymax}

        \pgfmathsetmacro{\xArel}{#1}
        \pgfmathsetmacro{\yArel}{#3}
        \pgfmathsetmacro{\xBrel}{#1-#2}
        \pgfmathsetmacro{\yBrel}{\yArel}
        \pgfmathsetmacro{\xCrel}{\xArel}

        \pgfmathsetmacro{\lnxB}{\xmin*(1-(#1-#2))+\xmax*(#1-#2)} % in [xmin,xmax].
        \pgfmathsetmacro{\lnxA}{\xmin*(1-#1)+\xmax*#1} % in [xmin,xmax].
        \pgfmathsetmacro{\lnyA}{\ymin*(1-#3)+\ymax*#3} % in [ymin,ymax].
        \pgfmathsetmacro{\lnyC}{\lnyA+#4*(\lnxA-\lnxB)}
        \pgfmathsetmacro{\yCrel}{\lnyC-\ymin)/(\ymax-\ymin)} % THE IMPROVED EXPRESSION WITHOUT 'DIMENSION TOO LARGE' ERROR.

        \coordinate (A) at (rel axis cs:\xArel,\yArel);
        \coordinate (B) at (rel axis cs:\xBrel,\yBrel);
        \coordinate (C) at (rel axis cs:\xCrel,\yCrel);

        \draw[#5]   (A)-- node[pos=0.5,anchor=north] {1}
                    (B)-- 
                    (C)-- node[pos=0.5,anchor=west] {#4}
                    cycle;
    }
}
\newcommand{\logLogSlopeTriangleNeg}[5]
{
    \pgfplotsextra
    {
        \pgfkeysgetvalue{/pgfplots/xmin}{\xmin}
        \pgfkeysgetvalue{/pgfplots/xmax}{\xmax}
        \pgfkeysgetvalue{/pgfplots/ymin}{\ymin}
        \pgfkeysgetvalue{/pgfplots/ymax}{\ymax}

        \pgfmathsetmacro{\xArel}{#1}
        \pgfmathsetmacro{\yArel}{#3}
        \pgfmathsetmacro{\xBrel}{#1-#2}
        \pgfmathsetmacro{\yBrel}{\yArel}
        \pgfmathsetmacro{\xCrel}{\xArel}

        \pgfmathsetmacro{\lnxB}{\xmin*(1-(#1-#2))+\xmax*(#1-#2)} % in [xmin,xmax].
        \pgfmathsetmacro{\lnxA}{\xmin*(1-#1)+\xmax*#1} % in [xmin,xmax].
        \pgfmathsetmacro{\lnyA}{\ymin*(1-#3)+\ymax*#3} % in [ymin,ymax].
        \pgfmathsetmacro{\lnyC}{\lnyA+#4*(\lnxA-\lnxB)}
        \pgfmathsetmacro{\yCrel}{\lnyC-\ymin)/(\ymax-\ymin)} % THE IMPROVED EXPRESSION WITHOUT 'DIMENSION TOO LARGE' ERROR.

        \coordinate (A) at (rel axis cs:\xArel,\yArel);
        \coordinate (B) at (rel axis cs:\xBrel,\yBrel);
        \coordinate (C) at (rel axis cs:\xCrel,\yCrel);

        \draw[#5]   (A)-- node[pos=.5,anchor=south] {1}
                    (B)-- 
                    (C)-- node[pos=0.5,anchor=west] {#4}
                    cycle;
    }
}
\newcommand{\logLogSlopeTriangleFlipNeg}[5]
{
    \pgfplotsextra
    {
        \pgfkeysgetvalue{/pgfplots/xmin}{\xmin}
        \pgfkeysgetvalue{/pgfplots/xmax}{\xmax}
        \pgfkeysgetvalue{/pgfplots/ymin}{\ymin}
        \pgfkeysgetvalue{/pgfplots/ymax}{\ymax}

        \pgfmathsetmacro{\xBrel}{#1-#2}
        \pgfmathsetmacro{\yBrel}{#3}
        \pgfmathsetmacro{\xCrel}{#1}

        \pgfmathsetmacro{\lnxB}{\xmin*(1-(#1-#2))+\xmax*(#1-#2)} % in [xmin,xmax].
        \pgfmathsetmacro{\lnxA}{\xmin*(1-#1)+\xmax*#1} % in [xmin,xmax].
        \pgfmathsetmacro{\lnyA}{\ymin*(1-#3)+\ymax*#3} % in [ymin,ymax].
        \pgfmathsetmacro{\lnyC}{\lnyA+#4*(\lnxA-\lnxB)}
        \pgfmathsetmacro{\yCrel}{\lnyC-\ymin)/(\ymax-\ymin)} % THE IMPROVED EXPRESSION WITHOUT 'DIMENSION TOO LARGE' ERROR.

	\pgfmathsetmacro{\xArel}{\xBrel}
        \pgfmathsetmacro{\yArel}{\yCrel}

        \coordinate (A) at (rel axis cs:\xArel,\yArel);
        \coordinate (B) at (rel axis cs:\xBrel,\yBrel);
        \coordinate (C) at (rel axis cs:\xCrel,\yCrel);

        \draw[#5]   (A)-- node[pos=0.5,anchor=east] {#4}
                    (B)-- 
                    (C)-- node[pos=0.5,anchor=north] {1}
                    cycle;
    }
}
\newcommand{\logLogSlopeTriangleFlip}[5]
{
    \pgfplotsextra
    {
        \pgfkeysgetvalue{/pgfplots/xmin}{\xmin}
        \pgfkeysgetvalue{/pgfplots/xmax}{\xmax}
        \pgfkeysgetvalue{/pgfplots/ymin}{\ymin}
        \pgfkeysgetvalue{/pgfplots/ymax}{\ymax}

        \pgfmathsetmacro{\xBrel}{#1-#2}
        \pgfmathsetmacro{\yBrel}{#3}
        \pgfmathsetmacro{\xCrel}{#1}

        \pgfmathsetmacro{\lnxB}{\xmin*(1-(#1-#2))+\xmax*(#1-#2)} % in [xmin,xmax].
        \pgfmathsetmacro{\lnxA}{\xmin*(1-#1)+\xmax*#1} % in [xmin,xmax].
        \pgfmathsetmacro{\lnyA}{\ymin*(1-#3)+\ymax*#3} % in [ymin,ymax].
        \pgfmathsetmacro{\lnyC}{\lnyA+#4*(\lnxA-\lnxB)}
        \pgfmathsetmacro{\yCrel}{\lnyC-\ymin)/(\ymax-\ymin)} % THE IMPROVED EXPRESSION WITHOUT 'DIMENSION TOO LARGE' ERROR.

	\pgfmathsetmacro{\xArel}{\xBrel}
        \pgfmathsetmacro{\yArel}{\yCrel}

        \coordinate (A) at (rel axis cs:\xArel,\yArel);
        \coordinate (B) at (rel axis cs:\xBrel,\yBrel);
        \coordinate (C) at (rel axis cs:\xCrel,\yCrel);

        \draw[#5]   (A)-- node[pos=0.5,anchor=east] {#4}
                    (B)-- 
                    (C)-- node[pos=0.5,anchor=south] {1}
                    cycle;
    }
}
\newcommand{\td}[2]{\frac{{\rm d}#1}{{\rm d}{\rm #2}}}
\newcommand{\pd}[2]{\frac{\partial#1}{\partial#2}}

\newcommand{\nor}[1]{\left\| #1 \right\|}

\newcommand{\LRp}[1]{\left( #1 \right)}
\newcommand{\LRs}[1]{\left[ #1 \right]}

\newcommand{\LRb}[1]{\left| #1 \right|}
\newcommand{\LRc}[1]{\left\{ #1 \right\}}

\newcommand{\Grad} {\ensuremath{\nabla}}
\newcommand{\Div} {\ensuremath{\nabla\cdot}}

\newcommand{\jump}[1] {\ensuremath{\LRs{\![#1]\!}}}
\newcommand{\avg}[1] {\ensuremath{\LRc{\!\{#1\}\!}}}

\newcommand{\Dhat}{\widehat{D}}
\newcommand{\Lhat}{L^2\LRp{\Dhat}}

\newcommand{\Oh}{\Omega_h}

\newcommand{\eval}[2][\right]{\relax
  \ifx#1\right\relax \left.\fi#2#1\rvert}

\newcommand{\LinfDk}{L^{\infty}\LRp{D^k}}

\newcommand{\diag}[1]{{\rm diag}\LRp{#1}}

\newcolumntype{C}[1]{>{\centering\let\newline\\\arraybackslash\hspace{0pt}}m{#1}}

\makeatletter
\renewcommand\d[1]{\mspace{6mu}\mathrm{d}#1\@ifnextchar\d{\mspace{-3mu}}{}}
\makeatother

\date{}
\author{Jesse Chan, Russell J.\ Hewett, T.\ Warburton}
\title{Weight-adjusted discontinuous Galerkin methods: curvilinear meshes}

\begin{document}

\maketitle

\begin{abstract}
Traditional time-domain discontinuous Galerkin (DG) methods result in large storage costs at high orders of approximation due to the storage of dense elemental matrices.  In this work, we propose a weight-adjusted DG (WADG) methods for curvilinear meshes which reduce storage costs while retaining energy stability.  \textit{A priori} error estimates show that high order accuracy is preserved under sufficient conditions on the mesh, which are illustrated through convergence tests with different sequences of meshes.  Numerical and computational experiments verify the accuracy and performance of WADG for a model problem on curved domains.  
\end{abstract}

%\tableofcontents

\section{Introduction}

Accurate simulations of wave propagation in heterogeneous media and complex geometries are important to applications in seismology, electromagnetics, and engineering design.  High order methods are advantageous for wave propagation, as they result in low numerical dispersion and higher accuracy per-degree of freedom (degree of freedom), while unstructured mesh methods are applicable to irregular domains.  Finite element methods are well-suited to such purposes, as they can accomodate both unstructured meshes and arbitrarily high order approximations.  In this work, we specialize to high order time-domain discontinuous Galerkin (DG) methods on unstructured meshes, which are commonly used in time-domain wave propagation problems governed by hyperbolic partial differential equations (PDEs) \cite{hesthaven2007nodal, karniadakis2013spectral}.  

High order DG methods with explicit timestepping are well-suited for hyperbolic problems; however, they result in a large number of floating point operations.  Kl{\"o}ckner et al.\ addressed this nontrivial cost by developing an implementation of DG on Graphics Processing Units (GPUs) in \cite{klockner2009nodal}, where the structure of time-explicit DG was found to map extremely well to the coarse and fine grained parallelism present in accelerator architectures.  Large problem sizes can be accomodated by parallelizing over multiple GPUs, which has been shown to yield scalable and efficient solvers \cite{godel2010scalability,modave2015nodal}.  

A limitation of most efficient implementations of DG methods is the assumption of affine simplicial elements, under which efficient quadrature-free DG implementations can be derived based only on reference element matrices.  When using high order methods for hyperbolic problems, numerous studies have shown that methods are limited to lower order accuracy if curved domain boundaries are approximated using planar tetrahedral elements \cite{wang2009discontinuous, zhang2015simple, zhang2016curved}.  High order accuracy is restored by accounting for the curvature of the boundary using high order accurate elemental mappings.  The most commonly used mappings are isoparametric \cite{persson2009curved, toulorge2016optimizing}, where the order of the mapping matches the order of approximation.  Rational \textit{isogeometric} mappings (which exactly represent geometries resulting from engineering design) have also been explored \cite{cottrell2009isogeometric, michoski2015foundations, engvall2016isogeometric, engvall2016isogeometricUnstructured}.  

On curvilinear elements, Jacobians and geometric factors vary spatially.  As a result, implementations of DG for affine elements lose high order accuracy or energy stability when naively applied to curvilinear elements.  These spatially varying Jacobians and geometric factors increase the computational cost of evaluating the DG variational formulation.  Spatially varying Jacobians are also embedded into elemental mass matrices, necessitating the precomputation and storage of dense matrices over each element.\footnote{An alternative to precomputation and storage of mass matrix factorizations is an on-the-fly solution of local matrix systems; however, such an approach is typically too expensive in the context of time-explicit methods.}  

On unstructured hexahedral meshes, Spectral Element Methods (SEM) \cite{patera1984spectral, canuto2012spectral} provide efficient ways to sidestep additional storage costs.  By employing high order Lagrange bases defined at tensor product Gauss-Legendre-Lobatto quadrature points, the mass matrix is well-approximated by a diagonal (lumped) mass matrix, which can be inverted and stored without increasing asymptotic storage costs.  The high order convergence of SEM is well-studied; see for example \cite{karniadakis2013spectral, kopriva2009implementing, canuto2012spectral}.  A drawback of using tensor-product elements is that the automatic generation of unstructured hexahedral meshes is, at present, infeasible on complex geometries \cite{shepherd2008hexahedral}.  

In contrast to the hexahedral case, there exists theory \cite{shewchuk1998tetrahedral}, algorithms, and software \cite{geuzaine2009gmsh} concerning the robust automatic generation of tetrahedral meshes.  However, extending mass-lumped schemes for time-domain continuous Galerkin methods to triangles and tetrahedra is less straightforward than for tensor product hexahedra.  Chin-Joe-Kong et al.\ \cite{chin1999higher} and Cohen et al.\ \cite{cohen2001higher} explicitly construct high order nodes which enable mass lumping on triangles and tetrahedra.  These sets contain nodes which are distributed topologically (vertex, edges, faces, interior); however, the number of nodes exceeds the cardinality of natural approximation spaces on simplices.  Additionally, such nodal sets have only been constructed up to degree $4$ for tetrahedra.  

A similar approach is taken for \textit{flux reconstruction} schemes on simplices, which are closely related to filtered nodal DG methods \cite{allaneau2011connections, zwanenburg2016equivalence} where nodes are taken to be unisolvent quadrature points \cite{williams2014symmetric, witherden2015identification}.  Unlike nodal sets for mass-lumped simplices, these quadrature points do not contain nodes which lie on the boundary, necessitating an additional interpolation step in the computation of numerical fluxes.  However, numerical evidence indicates that co-locating nodes and quadrature points reduces instabilities resulting from the aliasing of spatially varying Jacobians \cite{witherden2015development}, though an analysis of high order convergence and energy stability for curvilinear simplices are open problems.  

Krivodonova and Berger introduced an inexpensive treatment of curved boundaries for two-dimensional flow problems by modifying the DG formulation on affine triangles \cite{krivodonova2006high}. This was extended to wave propagation problems by Zhang in \cite{zhang2016curved}, and by Zhang and Tan for elements with non-boundary curved faces in \cite{zhang2015simple}.  A theoretical stability and convergence analysis remains to be shown, though numerical results suggest that each of these approaches preserves stability and high order accuracy on curvilinear meshes under the condition that curved triangles are well-approximated by planar triangles.  However, sufficiently large differences between curved and planar triangles still result in unstable schemes \cite{zhang2016curved}.  

%Flux reconstruction methods on tensor-product elements behave very similarly to spectral element DG methods (DG-SEM) \cite{de2014connections, chan2015gpu}, though a theoretical analysis of stability and high order convergence 

An alternative treatment addressing increased storage costs of curvilinear DG was addressed by Warburton using the Low-Storage Curvilinear DG (LSC-DG) method \cite{warburton2010low,warburton2013low}.  Under LSC-DG, the spatial variation of the Jacobian is incorporated into the physical basis functions over each element, resulting in identical mass matrices over each element.  Work in \cite{warburton2013low} also includes \textit{a priori} estimates for projection errors under the LSC-DG basis, and gives sufficient conditions under which convergence is guaranteed.  Furthermore, the DG variational formulation is constructed to be \textit{a priori} stable for surface quadratures with positive weights, allowing for stable under-integration of high order integrands present for curvilinear elements.  

In \cite{chan2016weight}, the weight-adjusted DG (WADG) method was introduced for wave propagation in heterogeneous media.  This WADG formulation results in a low storage method which is both energy stable and provably high order accurate in the presence of smoothly varying material data.  We extend these results to curvilinear meshes and demonstrate several advantages of WADG over LSC-DG.  Theoretical results for WADG in \cite{chan2016weight} are generalized to curvilinear elements, and a computational performance analysis on GPUs is presented.  We restrict ourselves to isoparametric mappings in this work, though the method and theory are readily extended to non-polynomial mappings. %\footnote{Isoparametric properties are required for the stability of one of two presented variational formulations; however, the latter variational formulation presented does not require isoparametric mappings to guarantee energy stability. }  

The structure of the paper-is as follows: section~\ref{sec:wadg_ip} introduces the concept of a weight-adjusted approximation to a weighted $L^2$ inner product, as well as \textit{a priori} estimates for curvilinear elements.  Section~\ref{sec:wadg} introduces a low-storage weight-adjusted DG method (WADG) for curvilinear meshes, along with two energy stable DG formulations.  Section~\ref{sec:num} presents numerical experiments which verify theoretical results and compare the behavior of WADG to LSC-DG. Section~\ref{sec:comp} presents a computational performance analysis of WADG on a single GPU.  Finally, section~\ref{sec:het} describes how to extend the curvilinear WADG method to wave propagation in heterogeneous media in a manner which is both high order accurate and computationally efficient.  

\section{Mathematical notation}
\label{sec:notation}

%\note{Make more distinct from first WADG paper.}

We assume a domain $\Omega$ which is exactly represented by a triangulation $\Omega_h$ consisting of (possibly curved) elements $D^k$.  We further assume that $D^k$ is the image of a reference element under the mapping 
\[
\bm{x}^k = \bm{\Phi}^k \widehat{\bm{x}},
\]
where $\widehat{\bm{x}} = \LRc{\widehat{x},\widehat{y},\widehat{z}}$ are coordinates on the reference element, $\bm{x}^k = \LRc{x^k,y^k,z^k}$ are physical coordinates on the $k$th element, and $J^k$ is the Jacobian of the transformation $\bm{\Phi}^k$.  

Over each element $D^k \in \Oh$, the solution is approximated from within the space $V_h\LRp{D^k}$
\[
V_h\LRp{D^k} = \bm{\Phi}^k \circ V_h\LRp{\widehat{D}},
\]
where $V_h\LRp{\widehat{D}}$ is an approximation space over the reference element.  The global approximation space is taken to be the direct sum of approximation spaces over each element, 
\[
V_h\LRp{\Omega_h} = \bigoplus_{D^k} V_h\LRp{D^k}.
\]
In this work, $\widehat{D}$ is taken to be the bi-unit quadrilateral $\widehat{D} = [-1,1]^2$ or the bi-unit right triangle, 
\[
\widehat{D} = \LRc{ -1 \leq  \widehat{{x}}, \widehat{y}; \quad \widehat{x}+\widehat{y}\leq 0},
\]
in two dimensions. In three dimensions, $\widehat{D}$ is the bi-unit right tetrahedron,
\[
\widehat{D} = \LRc{ -1 \leq  \widehat{{x}}, \widehat{y}, \widehat{z}; \quad \widehat{x}+\widehat{y}+\widehat{z} \leq -1}.  
\]
On the reference quadrilateral, $V_h\LRp{\widehat{D}}$ is the space of maximum degree $N$ polynomials, 
\[
V_h\LRp{\widehat{D}} = Q^N\LRp{\widehat{D}} = \LRc{ \widehat{x}^i \widehat{y}^j, \quad 0 \leq i,j \leq N}.  
\]
On the reference triangle, $V_h\LRp{\widehat{D}}$ is taken to be the space of total degree $N$ polynomials, 
\[
V_h\LRp{\widehat{D}} = P^N\LRp{\widehat{D}} = \LRc{ \widehat{x}^i \widehat{y}^j, \quad 0 \leq i + j \leq N},
\]
while on the reference tetrahedron, $V_h\LRp{\widehat{D}}$ is taken to be
\[
V_h\LRp{\widehat{D}} = P^N\LRp{\widehat{D}} = \LRc{ \widehat{x}^i \widehat{y}^j \widehat{z}^k, \quad 0 \leq i + j + k \leq N}.
\]
We define $\Pi_N$ as the $L^2$ projection onto $P^N\LRp{\widehat{D}}$ such that
\[
\LRp{\Pi_N u,v}_{L^2\LRp{\widehat{D}}} = \LRp{u,v}_{L^2\LRp{\widehat{D}}}, \qquad v\in P^N\LRp{\widehat{D}},
\]
where $\LRp{\cdot,\cdot}_{L^2\LRp{\widehat{D}}}$ denotes the $L^2$ inner product over the reference element. 

We also introduce standard Lebesgue $L^p$ norms over a general domain $\Omega$,
\begin{align*}
\nor{u}_{L^p\LRp{\Omega}} &= \LRp{\int_{\Omega} u^p}^{1/p} \qquad 1 \leq p < \infty, \\
\nor{u}_{L^{\infty}\LRp{\Omega}} &= \inf\LRc{C \geq 0: \LRb{u\LRp{\bm{x}}} \leq C \quad \forall \bm{x}\in \Omega}.
\end{align*}
Additionally, the $L^p$ Sobolev seminorms and norms of degree $s$ are defined as
\[
\LRb{u}_{W^{s,p}\LRp{\Omega}} = \LRp{\sum_{\LRb{\alpha}= s} \nor{ D^{\alpha} u}_{L^p\LRp{\Omega}}^p}^{1/p}, \qquad \LRb{u}_{W^{s,\infty}\LRp{\Omega}} = \max_{\LRb{\alpha}= s} \nor{D^{\alpha}u}_{L^{\infty}\LRp{\Omega}}, 
\]
and
\[
\nor{u}_{W^{s,p}\LRp{\Omega}} = \LRp{\sum_{\LRb{\alpha}\leq s} \nor{ D^{\alpha} u}_{L^p\LRp{\Omega}}^p}^{1/p}, \qquad \nor{u}_{W^{s,\infty}\LRp{\Omega}} = \max_{\LRb{\alpha}\leq s} \nor{D^{\alpha}u}_{L^{\infty}\LRp{\Omega}},
\]
where $\alpha = \LRc{\alpha_1,\alpha_2,\alpha_3}$ is a multi-index such that
\[
D^{\alpha}u = \pd{^{\alpha_1}}{x^{\alpha_1}}\pd{^{\alpha_2}}{y^{\alpha_2}}\pd{^{\alpha_3}}{z^{\alpha_3}} u.
\]

\section{Weight-adjusted approximations of weighted $L^2$ inner products}
\label{sec:wadg_ip}

In this section, we briefly review weighted $L^2$ inner products and weight-adjusted approximations.  Assuming a bounded positive weight $w$, 
\[
0 < w_{\min} \leq w \leq w_{\max} < \infty, 
\]
it is possible to define a weighted $L^2$ inner product over the reference element $\widehat{D}$,
\begin{equation}
\LRp{u,v}_w \coloneqq \int_{\widehat{D}} uvw.
\label{eq:wip}
\end{equation}
Additionally, we introduce the operator $T_w: L^2\LRp{\widehat{D}} \rightarrow P^N\LRp{\widehat{D}}$ as
\[
T_w u = \Pi_N\LRp{wu}.
\]
It is straightforward to show that $T_w$ is self-adjoint and positive definite \cite{chan2016weight}.  We also define an operator $T_w^{-1}$ which approximates division by $w$
\[
T_w^{-1}: L^2\LRp{\widehat{D}} \rightarrow P^N\LRp{\widehat{D}}, \qquad \LRp{w T_w^{-1}u,v}_{\widehat{D}} = \LRp{u,v}_{\widehat{D}}.
\]
The weight-adjusted approximation to the weighted $L^2$ inner product (\ref{eq:wip}) is then defined as
\[
\LRp{u,v}_{T^{-1}_{1/w}} \coloneqq \LRp{T^{-1}_{1/w}u,v}_{L^2\LRp{\widehat{D}}}.
\]
The accuracy of the above approximation relies on the observation that, under certain conditions, $T_w \approx T_{1/w}^{-1}$.  Estimates for weight-adjusted approximations rely on a estimates for a weighted projection \cite{warburton2013low}:
\begin{theorem}[Theorem 3.1 in \cite{warburton2013low}]
\label{thm:wproj}
Let $D^k$ be a quasi-regular element with representative size $h = {\rm diam}\LRp{D^k}$.  For $N \geq 0$, $w\in W^{N+1,\infty}\LRp{D^k}$, and $u\in W^{N+1,2}\LRp{D^k}$, 
\[
\nor{u - \frac{1}{w} \Pi_N\LRp{{u}{w}}}_{L^2\LRp{D^k}} \leq C h^{N+1}\nor{\frac{1}{\sqrt{J^k}}}_{\LinfDk}\nor{\frac{\sqrt{J^k}}{w}}_{\LinfDk} \nor{w}_{W^{N+1,\infty}\LRp{D^k}}  \nor{u}_{W^{N+1,2}\LRp{D^k}}.
\]
\end{theorem}
We assume that $\Omega_h$ is a mesh of quasi-regular elements for the remainder of the paper.  
%We introduced the following notation for convenience: 
%\begin{align*}
%\kappa_{\sqrt{J}} &\coloneqq \nor{\sqrt{J}}_{\LinfDk} \nor{\frac{1}{\sqrt{J}}}_{\LinfDk}\\
%\kappa_{J} &\coloneqq \nor{J}_{L^{\infty}} \nor{\frac{1}{J}}_{\LinfDk}\\
%\kappa_{w} &\coloneqq \nor{w}_{L^{\infty}} \nor{\frac{1}{w}}_{\LinfDk}.
%\end{align*}
%Estimates for moments of the difference between the weighted and weight-adjusted inner products were also given in \cite{chan2016weight}, which the following theorem generalizes:
%\begin{theorem}[Generalization of Theorem 6 in \cite{chan2016weight}]
%\label{thm:wmom}
%Let $u\in W^{N+1,2}\LRp{D^k}$, $w\in W^{N+1,\infty}\LRp{D^k}$, and $v \in P^M\LRp{D^k}$ for $0 \leq M\leq N$; then
%\begin{align*}
%\LRb{\LRp{{w}u,v}_{L^2\LRp{D^k}} - \LRp{u,v}_{w,D^k}} \leq C_{w,J} h^{2N+2 - M} \nor{w}_{W^{N+1,\infty}\LRp{D^k}} \nor{u}_{W^{N+1,2}\LRp{D^k}} \nor{v}_{L^{2}\LRp{D^k}},
%\end{align*}
%where the expression for $C_{w,J}$  is
%\[
%C_{w,J} \coloneqq  \nor{\frac{1}{J}}_{\LinfDk} \nor{\sqrt{J}}_{\LinfDk} \nor{\frac{\sqrt{J}}{w}}_{\LinfDk} \nor{w}_{\LinfDk}^2 \nor{\frac{1}{w}}_{\LinfDk}.
%\]
%\end{theorem}
%\begin{proof}
%blah
%\end{proof}
%

For curved meshes, we approximate the weighted $L^2$ inner product with the weight-adjusted inner product for $w = J$, where $J$ is the Jacobian and we have dropped the superscript by restricting ourselves to an arbitrary element $D^k$.  This results in the following weight-adjusted pseudo-projection problem 
\begin{align}
\LRp{P_Nu ,v}_{T^{-1}_{1/J}} = 
\LRp{T_{1/J}^{-1} P_Nu,v}_{\Lhat} =
\LRp{uJ,v}_{\Lhat}.
\label{eq:pseudoproj}
\end{align}
The solution to \ref{eq:pseudoproj} approximates the true $L^2$ projection, and can be given explicitly as follows: 
\begin{theorem}
The solution to the weight-adjusted pseudo-projection problem (\ref{eq:pseudoproj}) is
\[
P_Nu = \Pi_N\LRp{\frac{1}{J} \Pi_N\LRp{uJ}}.
\]
\end{theorem}
\begin{proof}
By the definition of $T_{1/J}^{-1}$ and $\Pi_N$, 
\[
\LRp{\frac{1}{J} T_{1/J}^{-1} \Pi_N\LRp{\frac{1}{J} \Pi_N\LRp{uJ}},v}_{\Lhat} = \LRp{\Pi_N\LRp{\frac{1}{J} \Pi_N\LRp{uJ}},v}_{\Lhat}  = \LRp{{\frac{1}{J} \Pi_N\LRp{uJ}},v}_{\Lhat}.
\]
which implies that 
\[
\LRp{{T_{1/J}^{-1} \Pi_N\LRp{\frac{1}{J} \Pi_N\LRp{uJ}} - \Pi_N\LRp{uJ}},\frac{v}{J}}_{\Lhat} = 0, \qquad \forall v\in P^N\LRp{\Dhat}.
\]
Since $T_{1/J}^{-1} \Pi_N\LRp{\frac{1}{J} \Pi_N\LRp{uJ}}$ and $ \Pi_N\LRp{uJ}$ are both polynomial, by a counting argument
\[
T_{1/J}^{-1} \Pi_N\LRp{\frac{1}{J} \Pi_N\LRp{uJ}}= \Pi_N\LRp{uJ},
\]
which satisfies (\ref{eq:pseudoproj}).
\end{proof}
This also yields the following \textit{a priori} bound: 
\begin{theorem}
\label{lemma:proj}
Let $u \in W^{N+1,2}\LRp{D^k}$.  Then, 
\[
\nor{u - P_N u}_{L^2\LRp{D^k}} \leq CC_J h^{N+1}\nor{u}_{W^{N+1,2}\LRp{D^k}}.
\]
where $C$ is a generic constant independent of $D^k$, and
\[
C_J =  \max\LRc{\nor{\frac{1}{\sqrt{J}}}_{L^{\infty}} \nor{\sqrt{J}}_{L^{\infty}}, \nor{\frac{1}{{J}}}_{L^{\infty}} \nor{J}_{W^{N+1,\infty}\LRp{D^k}} }.
\]
\end{theorem}
\begin{proof}
The triangle equality gives
\[
\nor{u - P_N u}_{L^2\LRp{D^k}} \leq \nor{u - \Pi_N u}_{L^2\LRp{D^k}}  + \nor{\Pi_Nu - P_N u}_{L^2\LRp{D^k}}.
\]
The former term is bounded by regularity assumptions on $u$,
\[
\nor{u - \Pi_N u}_{L^2\LRp{D^k}} \leq C_1 h^{N+1} \nor{\frac{1}{\sqrt{J}}}_{L^{\infty}} \nor{\sqrt{J}}_{L^{\infty}} \nor{u}_{W^{N+1,2}\LRp{D^k}},
\]
while the latter term is bounded by Theorem~\ref{thm:wproj} for $w = J$,
\begin{align*}
\nor{\Pi_Nu - P_N u}_{L^2\LRp{D^k}} &= \nor{\Pi_N \LRp{u - \frac{1}{J} \Pi_N\LRp{uJ}}}_{L^2\LRp{D^k}} \leq \nor{ {u - \frac{1}{J} \Pi_N\LRp{uJ}}}_{L^2\LRp{D^k}}\\
&\leq C_2 h^{N+1} \nor{\frac{1}{\sqrt{J}}}_{L^{\infty}}^2\nor{J}_{W^{N+1,\infty}\LRp{D^k}} \nor{u}_{W^{N+1,2}\LRp{D^k}}.
\end{align*}
where we have used that $\nor{\Pi_N}_{L^2\LRp{\widehat{D}}}=1$.
\end{proof}
We note that $C_J$ is taken to be the larger of two terms: $\nor{\frac{1}{\sqrt{J}}}_{L^{\infty}} \nor{\sqrt{J}}_{L^{\infty}}$ and $\nor{\frac{1}{{J}}}_{L^{\infty}} \nor{J}_{W^{N+1,\infty}\LRp{D^k}}$.  The former shows up in estimates for $L^2$ projection on curvilinear elements \cite{warburton2013low,michoski2015foundations}, while the latter term illustrates the additional effect of the smoothness of $J$ on the WADG pseudo-projection.  Since (\ref{eq:pseudoproj}) requires approximating $uJ$ by polynomials, the approximation power of $P^N$ is split between $u$ and $J$.  As noted in \cite{warburton2013low}, the dependence of $C_J$ on $\nor{J}_{W^{N+1,\infty}\LRp{D^k}}$ can be interpreted as the ``stealing'' of approximation power from $u$ by $J$.  
%Lemma~\ref{lemma:proj} also shows that a sufficient condition for high order convergence is that the constant $C_{J} = O(1)$ under mesh refinement.  

Finally, we have the following generalization of Theorem 6 in \cite{chan2016weight} to curvilinear elements: 
\begin{theorem}[Generalization of Theorem 6 in \cite{chan2016weight}]
\label{thm:wmom}
Let $u\in W^{N+1,2}\LRp{D^k}$, $w\in W^{N+1,\infty}\LRp{D^k}$, and $v \in P^M\LRp{D^k}$ for $0 \leq M\leq N$, then
\begin{align*}
\LRb{\LRp{{w}u,v}_{L^2\LRp{D^k}} - \LRp{u,v}_{T_{1/w}^{-1}}} \leq C_{w,J} h^{2N+2 - M} \nor{w}_{W^{N+1,\infty}\LRp{D^k}} \nor{u}_{W^{N+1,2}\LRp{D^k}} \nor{v}_{L^{\infty}\LRp{D^k}},
\end{align*}
where 
\[
C_{w,J} \coloneqq  \nor{\frac{1}{J}}_{\LinfDk} \nor{\sqrt{J}}_{\LinfDk} \nor{\frac{\sqrt{J}}{w}}_{\LinfDk} \nor{w}_{\LinfDk}^2 \nor{\frac{1}{w}}_{\LinfDk}.
\]
\end{theorem}
The proof of Theorem~\ref{thm:wmom} is a straightforward generalization of the proof of Theorem 6 in \cite{chan2016weight} to non-affine elements and spatially varying $J$.  Specifying to case of $w = J$ reveals that the constant $C_{w,J}$ reduces to
\[
C_{w,J} = \nor{\sqrt{J}}_{\LinfDk} \nor{\frac{1}{\sqrt{J}}}_{\LinfDk}\nor{J}^2_{\LinfDk} \nor{\frac{1}{J}}_{\LinfDk}^2 = \LRp{\nor{J}_{\LinfDk}\nor{\frac{1}{J}}_{\LinfDk}}^{2.5}.
\]

We note that Theorem~\ref{thm:wmom} is important because local conservation is not preserved under the use of a weight-approximated inner product \cite{chan2016weight}.  Taking $v=1$, Theorem~\ref{thm:wmom} implies that the local conservation error (the difference of the zeroth order moment of weighted and weight-adjusted $L^2$ inner products) converges at a rate of $2N+2$.  It is also straightforward to restore local conservation by projecting $J$ to $P^N\LRp{\widehat{D}}$ (within the weight-adjusted inner product) or through a specific rank one adjustment and the Shermann-Morrison formula \cite{chan2016weight}.  

Theorems~\ref{thm:wproj},~\ref{lemma:proj}, and \ref{thm:wmom} show that the weight-adjusted pseudo-projection over curvilinear elements behaves similarly to the $L^2$ projection under the condition that $J$ satisfies sufficient regularity conditions.  This similarity will be used to construct a weight-adjusted discontinuous Galerkin (WADG) method for meshes with curvilinear elements in the following section.  

\section{A weight-adjusted discontinuous Galerkin (WADG) method for curvilinear meshes}
\label{sec:wadg}
A general semi-discrete DG formulation may be given as
\[
\td{\bm{Q}}{t} = -\bm{M}^{-1}\bm{A}_h \bm{Q}, 
\]
where $\bm{Q}$ are the degrees of freedom for the discrete solution and $\bm{A}_h$ is a matrix resulting from the spatial discretization of a PDE.  Because the approximation space is discontinuous, $\bm{M}$ is block diagonal, with each block corresponding to $\bm{M}_{J^k}$, the local mass matrix (weighted by $J$) over $D^k$, 
\[
\LRp{\bm{M}_{J^k} }_{ij} = \int_{\widehat{D}} \phi_j\phi_i J,
\]
where $\phi_j,\phi_i$ are basis functions over $\widehat{D}$.  For affine elements, $J$ is constant over $D^k$, and each local mass matrix becomes a scaling of the reference mass matrix.  However, for curvilinear elements, each local mass matrix is distinct, and DG implementations must account for the factorization and storage of these inverses within the solver.  Because factorizing or storing local matrices requires $O(N^6)$ storage per-element (as opposed to $O(N^3)$ for the storage of degrees of freedom and geometric information), this greatly increases storage costs as $N$ increases.  

Weight-adjusted DG methods for curvilinear meshes address these storage costs by replacing the exact inversion of each block $\bm{M}_{J^k}^{-1}$ with the weight-adjusted approximation 
\[
\bm{M}_{J^k}^{-1} \approx \widehat{\bm{M}}^{-1} \bm{M}_{1/J^k} \widehat{\bm{M}}^{-1},
\]  
where $\widehat{\bm{M}}$ is the mass matrix over the reference element.  This is equivalent to replacing the global mass matrix $\bm{M}$ with the symmetric positive-definite weight-adjusted mass matrix $\tilde{\bm{M}}$, whose blocks are given by $\tilde{\bm{M}}^k \coloneqq \widehat{\bm{M}} \bm{M}^{-1}_{1/J^k} \widehat{\bm{M}}$.  Each block of the weight-adjusted mass matrix inverse may be applied in a matrix free fashion by assembling $\bm{M}_{1/J^k}$ using quadrature.  The application of the DG right hand side matrix $\bm{A}_h$ may also be applied in a matrix-free manner following standard techniques for isoparametric curved finite elements.  We employ quadrature constructed by Xiao and Gimbutas \cite{xiao2010quadrature} for evaluation of both volume and surface integrals.  

The semi-discrete WADG formulation is energy stable if $\bm{A}_h$ is weakly coercive such that $\bm{u}^T\bm{A}_h \bm{u} \geq 0$.  Multiplying the semi-discrete formulation by $\bm{u}^T\tilde{\bm{M}}$ on both sides then yields $$\frac{1}{2}\td{}{t} \bm{u}^T \tilde{\bm{M}} \bm{u} \leq 0,$$ implying that the $\tilde{\bm{M}}$-norm of the discrete solution does not increase in time.  We note that the precise form of $\tilde{\bm{M}}$ does not matter for energy stability, so long as it is positive definite.\footnote{The energy stability of the semi-discrete formulation also shows up in the derivation of \textit{a priori} semi-discrete error estimates \cite{grote2007interior, warburton2013low, chan2016weight}, where high order accuracy depends on equivalence of the discrete $\tilde{\bm{M}}$-norm with the $L^2$ norm shown in Section~\ref{sec:wadg_ip}.}  This implies that when evaluating the action of $\widehat{\bm{M}}^{-1} \bm{M}_{1/J^k} \widehat{\bm{M}}^{-1}$ in a matrix free manner, it is sufficient to take any quadrature strong enough to ensure that $\bm{M}_{1/J^k}$ is positive-definite.  A quadrature which is exact for polynomials of degree $2N$ is observed to be sufficient in practice for stability, though lower errors are reported for quadratures of degree $2N+1$ \cite{chan2016weight}.  In this work, we utilize quadrature of degree $2N+1$ in applying the inverse of the weight-adjusted mass matrix.  

%Because both LSC-DG and weight-adjusted DG are 
%It is worth constrasting the weight-adjusted DG method with LSC-DG \cite{warburton2013low}.  
The LSC-DG method \cite{warburton2013low} addresses storage costs for curvilinear meshes by introducing rational basis functions which incorporate the spatial variation of $J$.  Because terms in the variational formulation involving rational functions are no longer integrable using standard quadrature rules, energy stability requires the use of an \textit{a priori} stable quadrature-based formulation, where stability does not depend on the strength of the quadrature.  In contrast, weight-adjusted DG methods only modify the mass matrix, and can be paired with any basis and stable variational formulation to yield an energy-stable scheme.  

In the following sections, we present two energy stable weight-adjusted discontinuous Galerkin (WADG) formulations for the acoustic wave equation on meshes containing curvilinear elements.  These formulations are equivalent at the continuous level; however, they differ at the discrete level and in terms of computational cost.  

\subsection{Discrete variational formulations}
\label{sec:form}

In this work, we consider the propagation of acoustic waves, though the approaches can be extended to elastic or electromagnetic wave propagation as well.  The acoustic wave equation in first order form is given as 
\begin{align*}
\frac{1}{\rho c^2}\pd{p}{t}{} + \Div \bm{u} &= 0,\\
\rho\pd{\bm{u}}{t}{} + \Grad p &= 0,
\end{align*}
where $t$ is time, $p$ is pressure, $\bm{u}$ is the vector velocity, and $\rho$ and $c$ are density and wavespeed, respectively.  For now, we assume $\rho = 1$ and $c^2$ is constant over the domain $\Omega \in \mathbb{R}^3$, though we will generalize this later.  We additionally assume homogeneous Dirichlet boundary conditions $p=0$ on $\partial \Omega$.  

We introduce definitions of the jump and average of $u \in V_h\LRp{\Omega_h}$.  Let $f$ be a shared face between two elements $D^{k^-}$ and $D^{k^+}$, and let $u$ be a scalar functions.  The jump and average of $u$ are defined as
\[
\jump{u} = u^+ - u^-, \qquad \avg{u} = \frac{u^+ + u^-}{2}. %\qquad \jump{\bm{u}} = \bm{u}^+ - \bm{u}^-, \qquad \avg{\bm{u}} = \frac{\bm{u}^+ + \bm{u}^-}{2}.
\]
Jumps and averages of vector functions are defined in an analogous manner.  For faces $f$ which lie on the boundary $\partial \Omega$, homogeneous Dirichlet boundary conditions are enforced by defining the jumps of $p, \bm{u}$ through
\[
\left.p^+\right|_{f} = -\left.p^-\right|_{f}, \qquad \left.\bm{n}^+\bm{u}^+\right|_{f} = \left.\bm{n}^-\bm{u}^-\right|_{f}.
\]

We introduce first the so-called \textit{strong-weak} discontinuous Galerkin formulation.  Over an element $D^k$, this formulation is given locally as
\begin{align*}
\int_{D^k}\frac{1}{c^2}\pd{p}{t}{} v &= \int_{D^k} \bm{u}\cdot \Grad v - \int_{\partial D^k} \frac{1}{2}\LRp{2\avg{\bm{u}}\cdot \bm{n}^- - \tau_p\jump{p}} v, \\
\int_{D^k}\pd{\bm{u}}{t}{} \cdot \bm{\tau} &= - \int_{D^k} \Grad p \cdot \bm{\tau} - \int_{\partial D^k} \frac{1}{2}\LRp{ \jump{p} - \tau_u\jump{\bm{u}}\cdot\bm{n}^-}\bm{\tau}\cdot \bm{n}^-.
\end{align*}
where $\tau_p, \tau_u \geq 0$ are penalty parameters.  For $\tau_p = \tau_u = 1$, the above numerical flux is equivalent to an upwind flux for isotropic media, while for $\tau_p = \tau_u = 0$, the numerical flux reduces to a central flux.  

The strong-weak formulation is derived by multiplying the acoustic wave system by test functions $(v,\bm{\tau})$, integrating over the domain $\Omega$, then integrating by parts locally on each element $D^k$.  The pressure equation is integrated by parts once, while the velocity equation is integrated by parts twice.  This formulation was used in \cite{warburton2013low,chan2015gpu} to ensure energy stability when using inexact quadrature rules to evaluate the integrands.  Taking $v = p$ and $\bm{\tau} = \bm{u}$, summing up over all elements $D^k$, the volume integrals cancel.  Further rearrangement of flux terms yields
\[
\frac{1}{2} \pd{}{t} \int_{\Omega_h} \LRp{\frac{p^2}{c^2} + \LRb{\bm{u}}^2} \leq -\frac{1}{2}\sum_{f \in \Gamma_h} \int_f \tau_p\jump{p}^2 + \tau_u\LRp{\jump{\bm{u}}\cdot \bm{n}^-}^2,
\]
where $\Gamma_h$ is the collection of unique faces in $\Omega_h$.  This implies that the rate of change of the solution in time is non-increasing, and that for positive $\tau_p,\tau_u$, ``rough'' components of the solution with large jumps are dissipated in time.  We note that this statement still holds if volume and surface integrals are replaced by quadrature approximations, independently of the degree of the quadrature.\footnote{The strength of the quadrature used does impact behavior of the method.  However, as noted in \cite{warburton2013low}, quadrature strength impacts the accuracy but not energy stability.}  

The second formulation we consider is the \textit{strong} discontinuous Galerkin formulation, which is given locally as
\begin{align*}
\int_{D^k} \frac{1}{c^2}\pd{p}{t}{} v &= -\int_{D^k} \Div \bm{u}v - \int_{\partial D^k} \frac{1}{2}\LRp{\jump{\bm{u}}\cdot \bm{n}^- - \jump{p}} v, \\
\int_{D^k}\pd{\bm{u}}{t}{} \cdot \bm{\tau} &= - \int_{D^k} \Grad p \cdot \bm{\tau} - \int_{\partial D^k} \frac{1}{2}\LRp{ \jump{p} - \jump{\bm{u}}\cdot\bm{n}^-}\bm{\tau}\cdot \bm{n}^-.
\end{align*}
Unlike the strong-weak formulation, both equations in the strong DG formulation are integrated by parts twice.  While both formulations are equivalent under exact quadrature, the strong formulation results in a more computationally efficient structure \cite{hesthaven2007nodal}.  However, the stability of the strong formulation depends on the quadrature rules used, as well as the geometric mapping being polynomial.  

For energy stability, it is sufficient for the strong formulation to be equivalent to the strong-weak formulation at the discrete level.  For the acoustic wave equation, this requires that 
\[
-\int_{D^k} \Div \bm{u} v - \int_{\partial D^k} \frac{1}{2}\jump{\bm{u}}\cdot \bm{n}^-  v = \int_{D^k}  \bm{u} \cdot\Grad v - \int_{\partial D^k} \avg{\bm{u}}\cdot \bm{n}^-  v,
\]
which can be achieved by choosing quadrature rules which are exact for the above volume and surface integrands.  Then, integration by parts holds at the discrete level.  Integrands in these volume and surface integrals involve the integration of geometric factors and reference derivatives, such as 
\[
\int_{\widehat{D}} \pd{\bm{u}_1}{x}v  J = \int_{\widehat{D}} \LRp{\pd{\bm{u}_1}{\widehat{x}}\pd{\widehat{x}}{x} J +\pd{\bm{u}_1}{\widehat{y}}\pd{\widehat{y}}{x}  J+\pd{\bm{u}_1}{\widehat{z}}\pd{\widehat{z}}{x}  J} \bm{\tau}_1.
\]
where $\bm{u}_1,\bm{\tau}_1$ are the first components of $\bm{u},\bm{\tau}$. %, and $\bm{n}_x$ is the component of the normal vector in the $x$ direction.  
The product of the Jacobian with geometric factors is polynomial \cite{hesthaven2007nodal,karniadakis2013spectral}; for example, 
\[
\pd{\widehat{x}}{x}{} J = \pd{y}{\widehat{y}}\pd{z}{\widehat{z}} - \pd{z}{\widehat{y}}\pd{y}{\widehat{z}}.
\]
For an isoparametric mapping, $x, y,z \in P^N\LRp{\widehat{D}}$, implying that the product of geometric factors with Jacobians is degree $2(N-1)$ in three dimensions.  Because $v\in P^N \LRp{\widehat{D}}$, $\pd{\widehat{x}}{x} \in P^{N-1}\LRp{\widehat{D}}$, this requires a quadrature rule over tetrahedra which integrates polynomials of degree $4N-3$ exactly.  

For the surface flux, integrands are of the form
\[
\int_{\widehat{f}} \bm{u} \cdot\bm{n}^- v J^f.
\]
where $\widehat{f}$ is a reference triangular face.  Similar formulas for the product of normals and surface Jacobians imply that $n_x J^f \in P^{2N-2}$ \cite{hesthaven2007nodal, johnen2013geometrical}.  Because $\bm{u},v\in P^N$, integration of surface fluxes requires a quadrature of degree $4N-2$ on a triangular face.  

We note that taking volume and surface quadratures as described above are sufficient conditions for discrete energy stability under isoparametric mappings; however, numerical experiments indicate that these are not always necessary conditions.  For example, we have observed that, for a range of tested meshes, taking quadratures which are exact only for degree $2N+1$ polynomials still result in energy stable systems (all eigenvalues of the discretization matrix have non-positive real part), though there is no theoretical analysis supporting this.  

Finally, because both formulations are energy stable, a semi-discrete error analysis can be derived in a straightforward manner using tools from Section~\ref{sec:wadg_ip} and approaches detailed in \cite{warburton2013low,chan2016weight}.  For brevity, we omit these details in this work.  

\section{Numerical experiments}
\label{sec:num}
In this section, we present numerical experiments which demonstrate the stability and high order convergence of weight-adjusted discontinuous Galerkin methods for curved meshes.  

\subsection{Convergence rates for weight-adjusted $L^2$ projection} 

First we compare the $L^2$ errors and convergence rates obtained using an $L^2$ projection onto a standard polynomial basis, an $L^2$ projection onto a rational LSC-DG basis, and the WADG pseudo-projection onto a polynomial basis for different sequences of non-affine meshes.  Numerical results are presented for quadrilateral meshes with $Q^N$ approximation spaces and exact quadrature.  

\subsubsection{Asymptotically non-affine meshes}
\label{sec:arnold}
We consider a sequence of meshes constructed through self-similar refinement, as shown in Figure~\ref{fig:trapMesh}.  These meshes (which we will refer to such as Arnold-type meshes) are introduced by Arnold, Boffi, and Falk in \cite{arnold2002approximation} to demonstrate the loss of convergence observed for serendipity finite elements under non-affine mappings.  In \cite{warburton2013low}, it is shown that $L^2$ projection errors for the LSC-DG basis stall on such a sequence of meshes, which was attributed to the fact that the LSC-DG projection error is bounded by 
\[
\nor{u - \frac{1}{\sqrt{J}}\Pi_N\LRp{u\sqrt{J}}}_{L^2\LRp{D^k}} \leq C h^{N+1} \nor{\frac{1}{\sqrt{J}}}_{L^\infty\LRp{D^k}}\nor{\sqrt{J}}_{W^{N+1,\infty}\LRp{D^k}}\nor{u}_{W^{N+1,2}\LRp{D^k}}.
\]
For Arnold-type meshes, it can be shown that $J = h^2 + x_1 h$ and that
\[
\nor{\sqrt{J^k}}_{W^{N+1,\infty}\LRp{D^k}} \approx O\LRp{\frac{1}{h^{N+1}}},
\]
which results in an $O(1)$ bound on the LSC-DG projection error, independent of mesh size.  In contrast, as shown in Section~\ref{sec:wadg_ip}, the WADG pseudo-projection error is bounded by 
\[
\nor{u - P_Nu} \leq C h^{N+1} \nor{\frac{1}{J}}_{L^\infty\LRp{D^k}}\nor{J}_{W^{N+1,\infty}\LRp{D^k}}\nor{u}_{W^{N+1,2}\LRp{D^k}}.  
\]
Because $J$ is a linear polynomial, $\nor{J}_{W^{N+1,\infty}\LRp{D^k}} = \nor{J}_{W^{1,\infty}\LRp{D^k}} = O(h)$.  The form of $J$ also implies that $\nor{\frac{1}{J}}_{\LinfDk} = O(1/h^2)$.  Combining these gives an $O(1/h)$ bound, implying that at most a single order of convergence is lost for the WADG pseudo-projection, as exhibited by the numerical results in Figure~\ref{fig:arnoldmeshes}.  

\begin{figure}
\centering
\subfloat[Initial mesh]{
\includegraphics[width=.3\textwidth]{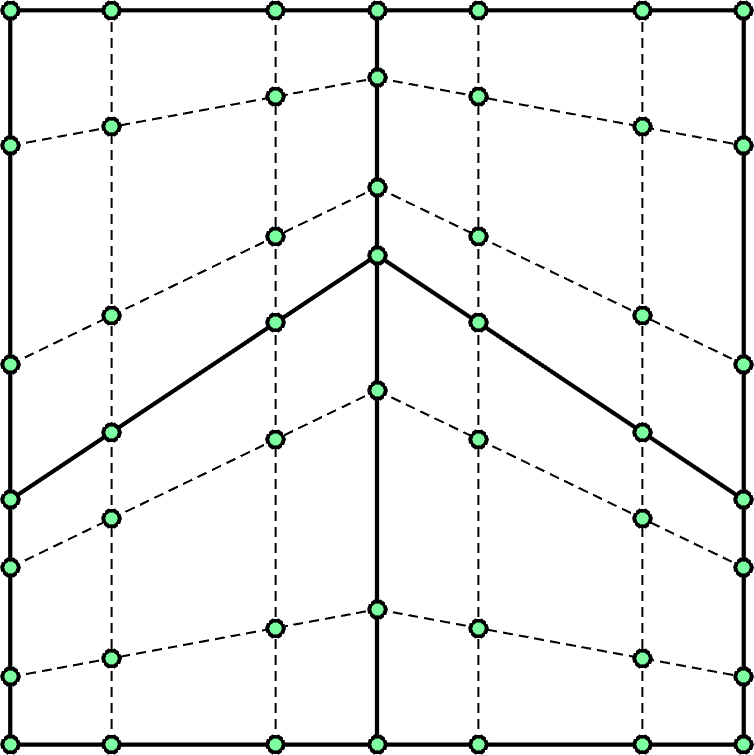}
}
\hspace{4em}
\subfloat[Refined mesh]{
\includegraphics[width=.3\textwidth]{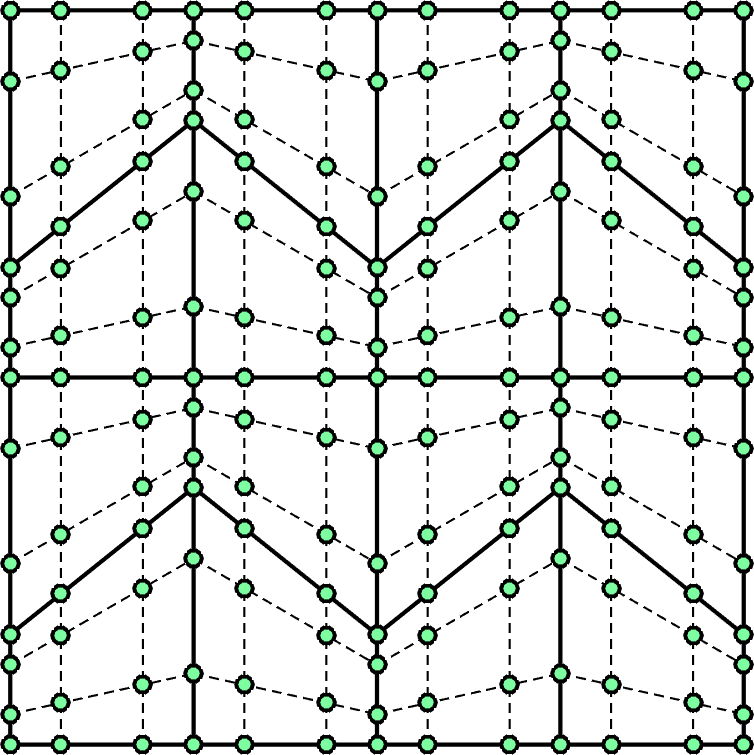}
}
\caption{Asymptotically non-affine Arnold-type meshes used in $h$-convergence studies.  High order nodes are represented by dots.  Meshes are shown for $N=3$.}
\label{fig:trapMesh}
\end{figure}

\begin{figure}
\centering
\subfloat[$N=3$]{
\begin{tikzpicture}
\begin{loglogaxis}[
	legend cell align=left,
	width=.425\textwidth,
    xlabel={Mesh size $h$},
    ylabel={$L^2$ error},
    xmin=.005, xmax=1,
    ymin=1e-13, ymax=1e-1,
    legend pos=south east,
    xmajorgrids=true,
    ymajorgrids=true,
    grid style=dashed,
] 
% adding this 2x because pgfplots dashes the line for some reason...
\addplot[color=blue,mark=*,mark size=3,semithick, mark options={fill=markercolor}]
coordinates{(0.5,0.00145465)(0.25,0.000140693)(0.125,1.17949e-05)(0.0625,8.58138e-07)(0.03125,5.78543e-08)(0.015625,3.75483e-09)};
\addplot[color=black,mark=diamond*,mark size=4,semithick, mark options={fill=markercolor}]
coordinates{(0.5,0.00159198)(0.25,0.000658818)(0.125,0.000619747)(0.0625,0.000677173)(0.03125,0.00073964)(0.015625,0.000782687)};
\addplot[color=red,mark=square*,mark size=3,semithick, mark options={fill=markercolor}]
coordinates{(0.5,0.00148852)(0.25,0.000176325)(0.125,2.60169e-05)(0.0625,3.94602e-06)(0.03125,5.58752e-07)(0.015625,7.47639e-08)};

\logLogSlopeTriangleFlip{0.35}{0.125}{0.55}{3}{red};
\logLogSlopeTriangle{0.35}{0.125}{0.325}{4}{blue};

\legend{$L^2$ projection, LSC-DG, WADG}
\end{loglogaxis}
\end{tikzpicture}
}
\subfloat[$N=4$]{
\begin{tikzpicture}
\begin{loglogaxis}[
	legend cell align=left,
	width=.425\textwidth,
    xlabel={Mesh size $h$},
    ylabel={$L^2$ error},
    xmin=.005, xmax=1,
    ymin=1e-13, ymax=1e-1,
    legend pos=south east,
    xmajorgrids=true,
    ymajorgrids=true,
    grid style=dashed,
] 
% adding this 2x because pgfplots dashes the line for some reason...
\addplot[color=blue,mark=*,mark size=3,semithick, mark options={fill=markercolor}]
coordinates{(0.5,9.41674e-05)(0.25,5.29694e-06)(0.125,2.39575e-07)(0.0625,9.03906e-09)(0.03125,3.10282e-10)(0.015625,1.01611e-11)};
\addplot[color=black,mark=diamond*,mark size=4,semithick, mark options={fill=markercolor}]
coordinates{(0.5,0.000129206)(0.25,8.56669e-05)(0.125,9.71846e-05)(0.0625,0.000117339)(0.03125,0.000133929)(0.015625,0.000144593)};
\addplot[color=red,mark=square*,mark size=3,semithick, mark options={fill=markercolor}]
coordinates{(0.5,9.75639e-05)(0.25,7.3955e-06)(0.125,6.4703e-07)(0.0625,5.22698e-08)(0.03125,3.78512e-09)(0.015625,2.55621e-10)};
\logLogSlopeTriangleFlip{0.325}{0.125}{0.325}{4}{red};
\logLogSlopeTriangle{0.325}{0.125}{0.09}{5}{blue};

\legend{$L^2$ projection, LSC-DG, WADG}
\end{loglogaxis}
\end{tikzpicture}
}
%\subfloat{\includegraphics[width=.5\textwidth]{}}
\caption{Errors for WADG pseudo-projection and $L^2$ projection with polynomial and LSC-DG bases on non-affine Arnold-type meshes. } %LSC-DG is observed to stall under $h$-refinement, while WADG pseudo-projection only loses on order of convergence in $h$.  }
\label{fig:arnoldmeshes}
\end{figure}
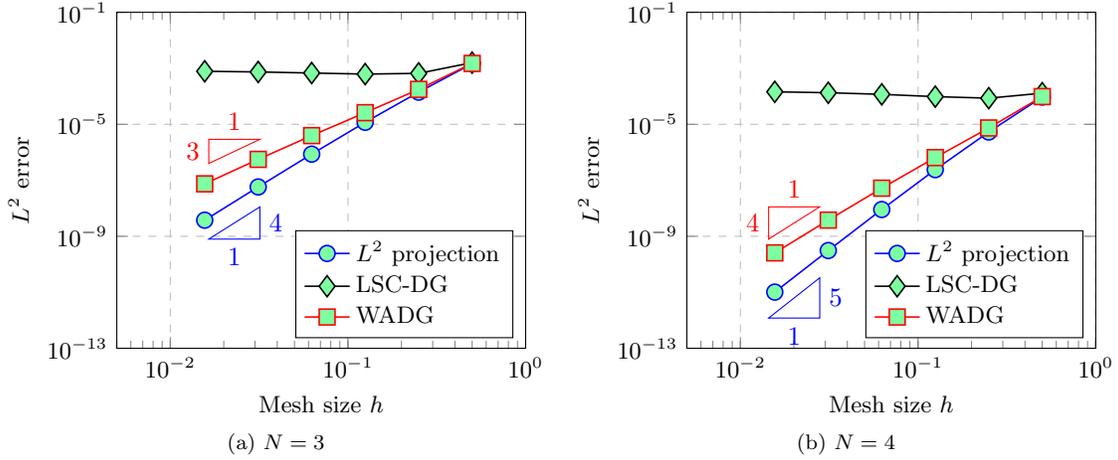

\subsubsection{Randomly perturbed curved meshes}

In this section, we compare the behavior of the WADG pseudo-projection with standard and LSC-DG projection for meshes with randomly generated curved perturbations.   These meshes are constructed by defining an elemental map in terms of nodal coordinates, perturbing these coordinates for non-curved meshes, and generating curvilinear mappings based on the resulting nodal distributions.  Botti noted in \cite{botti2012influence} that $Q^N\LRp{D^k}$ is not necessarily contained in the local finite element approximation space if $D^k$ is the image of a polynomial map with order greater than one.  This implies that, for a sequence of arbitrary curved meshes with $h\rightarrow 0$, optimal rates of convergence will not necessarily be observed.  Sufficient conditions are also described in \cite{warburton2013low, michoski2015foundations} under which optimal $L^2$ convergence rates are expected for non-affine mappings.  

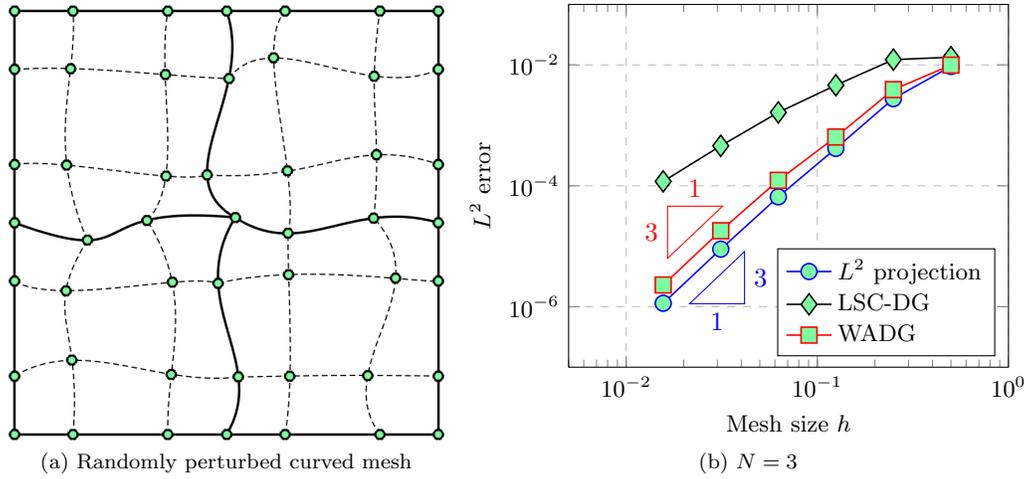
\begin{figure}
\centering
\subfloat[Randomly perturbed curved mesh]{
\label{subfig:randunif1}
\includegraphics[width=.35\textwidth]{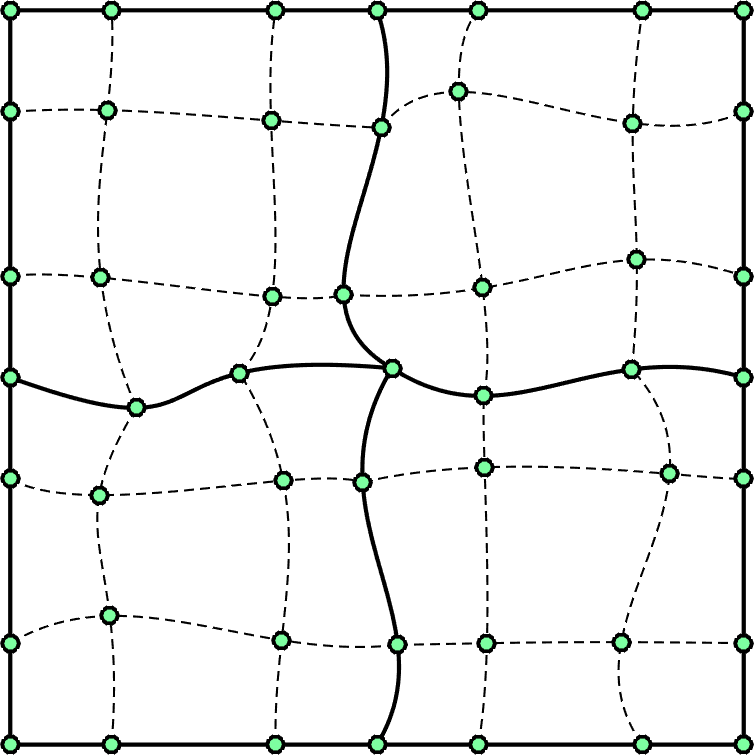}
}
\subfloat[$N=3$]{
\label{subfig:randunif2}
\begin{tikzpicture}
\begin{loglogaxis}[
	legend cell align=left,
	width=.45\textwidth,
    xlabel={Mesh size $h$},
    ylabel={$L^2$ error},
    xmin=.005, xmax=1,
    ymin=1e-7, ymax=1e-1,
    legend pos=south east,
    xmajorgrids=true,
    ymajorgrids=true,
    grid style=dashed,
] 
% adding this 2x because pgfplots dashes the line for some reason...
\addplot[color=blue,mark=*,mark size=3,semithick, mark options={fill=markercolor}]
coordinates{(0.5,0.00943375)(0.25,0.00276313)(0.125,0.000415615)(0.0625,6.56974e-05)(0.03125,8.95162e-06)(0.015625,1.14109e-06)};
\addplot[color=black,mark=diamond*,mark size=4,semithick, mark options={fill=markercolor}]
coordinates{(0.5,0.0134065)(0.25,0.012209)(0.125,0.004612)(0.0625,0.00163555)(0.03125,0.000458367)(0.015625,0.00011822)};
\addplot[color=red,mark=square*,mark size=3,semithick, mark options={fill=markercolor}]
coordinates{(0.5,0.00990795)(0.25,0.00389737)(0.125,0.00064227)(0.0625,0.000123078)(0.03125,1.80455e-05)(0.015625,2.28866e-06)};
%\node at (axis cs:.03,2.1e-07) {$a = 10^{-1}$};
\logLogSlopeTriangleFlip{0.35}{0.125}{0.3}{3}{red};
\logLogSlopeTriangle{0.4}{0.125}{0.175}{3}{blue};

\legend{$L^2$ projection, LSC-DG, WADG}
\end{loglogaxis}
\end{tikzpicture}
}

\caption{Curved mesh constructed by randomly perturbing nodal positions for a uniform mesh (\ref{subfig:randunif1}), and errors for WADG pseudo-projection and $L^2$ projection with polynomial and LSC-DG bases for $N=3$ (\ref{subfig:randunif2}). } %The $L^2$ projection and WADG pseudo-projection are observed to converge at the same (suboptimal) rate. }
\label{fig:randunif}
\end{figure}

We consider convergence on curved meshes constructed from random perturbations to uniform meshes of quadrilaterals, as shown in Figure~\ref{fig:randunif}.  As predicted,  $L^2$ projection no longer delivers optimal rates of convergence.  In contrast to LSC-DG, the $L^2$ errors for the WADG pseudo-projection converge at the same rate as the $L^2$ projection, as shown in Figure~\ref{fig:randunif}.  

\subsubsection{Asymptotically curved meshes}

Finally, we investigate convergence on isoparametric curved analogues of Arnold-type meshes, as shown in Figure~\ref{fig:randarnold}.  Like the straight-sided Arnold-type meshes in Section~\ref{sec:arnold}, self-similar refinement is used to generate a sequence of meshes with asymptotically non-affine elements.  We assume that the number of elements along both the $x$ and $y$ coordinates is equal to $K_{\rm 1D}$.  Then, for selective elements, the $y$ coordinate of the vertical face is displaced according to $dy$, where 
\[
dy(x) = \frac{\omega}{K_{\rm 1D}+1} \cos\LRp{\frac{K_{\rm 1D}}{2}\pi \LRp{x+1}}
\]
and $\omega \in [0,2]$ is a warping parameter.  The deformation is interpolated at boundary nodes and then smoothly blended into the interior of the element using a linear blending function, as shown in Figure~\ref{fig:meshwarping}.

\begin{figure}
\centering
\subfloat[Initial mesh,  $\omega = 2$]{
\includegraphics[width=.275\textwidth]{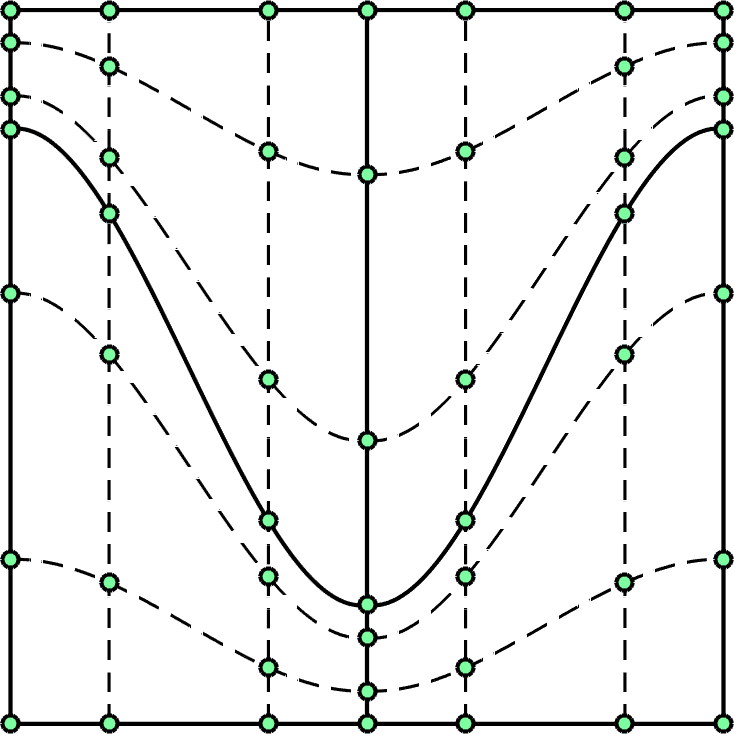}
}
%\hspace{2em}
\subfloat[Initial mesh,  $\omega = 1$]{
\includegraphics[width=.275\textwidth]{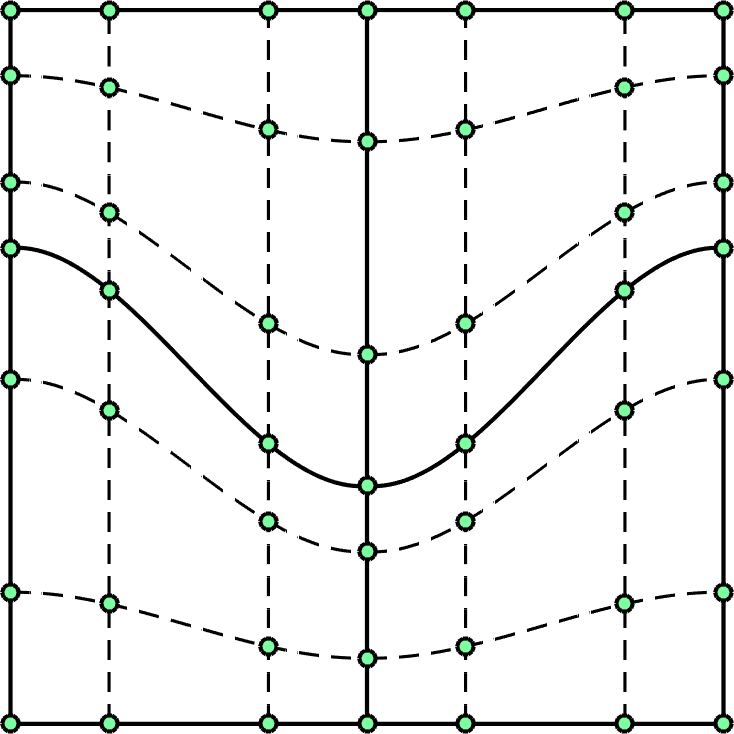}
}
%\hspace{2em}
\subfloat[Initial mesh,  $\omega = 1/4$]{
\includegraphics[width=.275\textwidth]{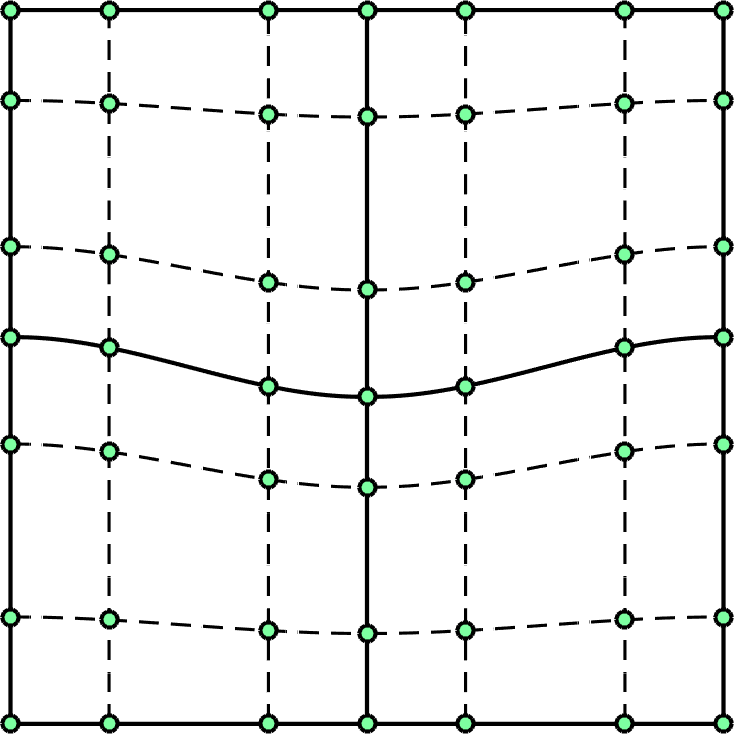}
}
\\
\subfloat[Refined mesh,  $\omega = 2$]{
\includegraphics[width=.275\textwidth]{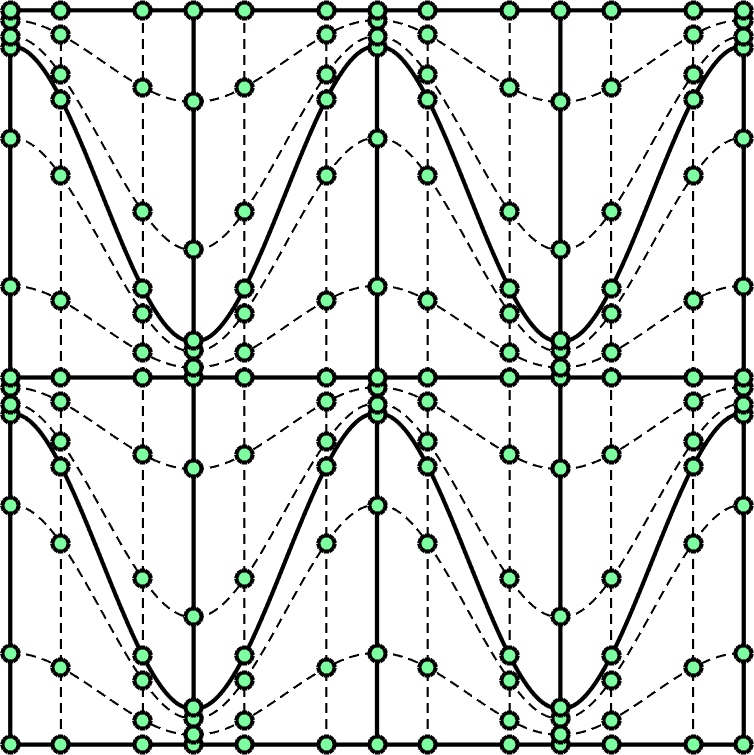}
}
%\hspace{2em}
\subfloat[Refined mesh,  $\omega = 1$]{
\includegraphics[width=.275\textwidth]{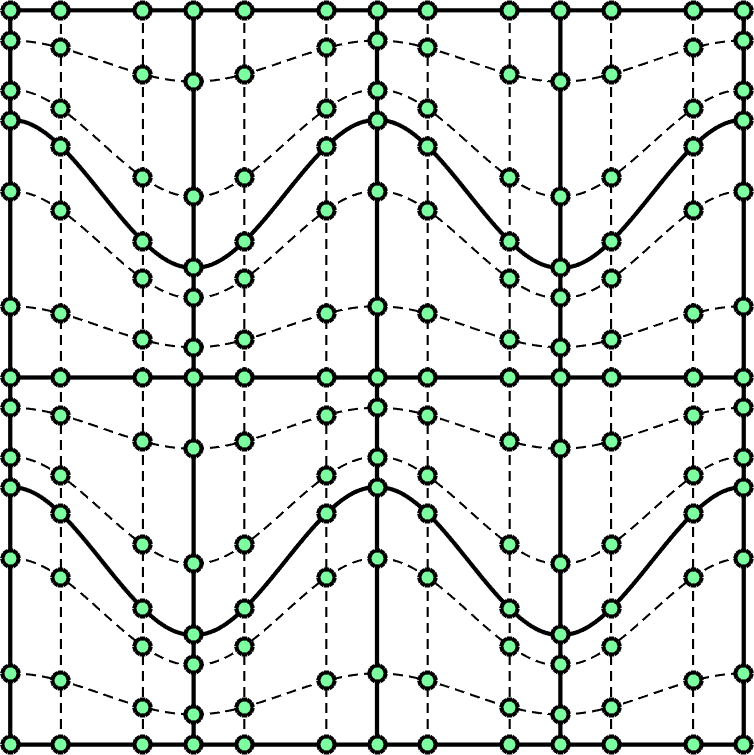}
}
%\hspace{2em}
\subfloat[Refined mesh,  $\omega = 1/4$]{
\includegraphics[width=.275\textwidth]{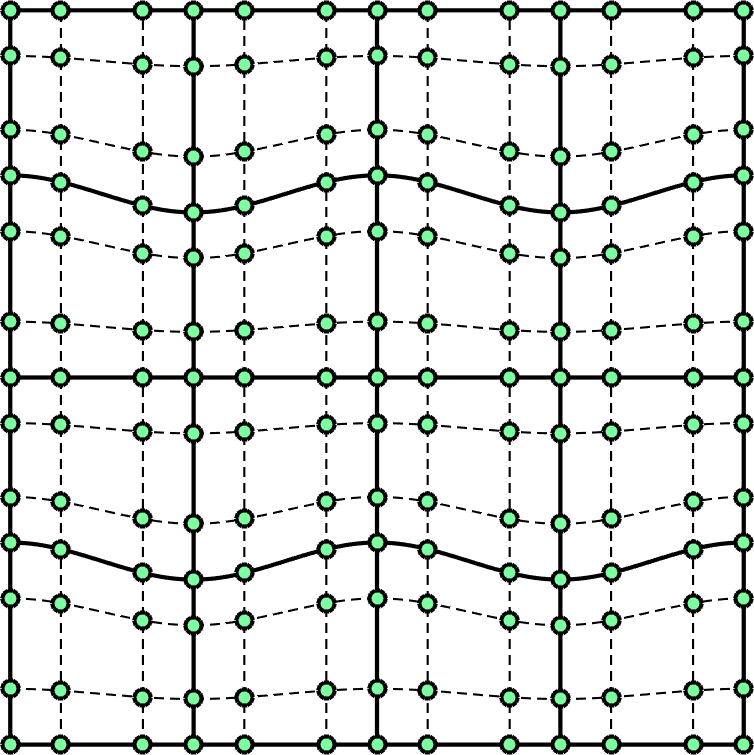}
}
\caption{Curvilinear analogues of Arnold-type meshes for three different warping parameters. }
\label{fig:meshwarping}
\end{figure}

\begin{figure}
\centering
%\subfloat[$N=3$]{
\begin{tikzpicture}
\begin{loglogaxis}[
legend style={font=\tiny},
	legend cell align=left,
	width=.425\textwidth,
    xlabel={Mesh size $h$},
    ylabel={$L^2$ error},
    xmin=.0025, xmax=1,
    ymin=1e2, ymax=1e11,
    legend pos=north east,
    xmajorgrids=true,
    ymajorgrids=true,
    grid style=dashed,
] 
\addplot[color=blue,mark=*,mark size=3,semithick, mark options={fill=markercolor}]
coordinates{(0.5,6099.2)(0.25,83523.1)(0.125,1.26205e+06)(0.0625,1.96906e+07)(0.03125,3.11385e+08)(0.015625,5.02342e+09)};
\addplot[color=black,mark=diamond*,mark size=4,semithick, mark options={fill=markercolor}]
coordinates{(0.5,1590.78)(0.25,13993.8)(0.125,126284)(0.0625,1.094e+06)(0.03125,9.15612e+06)(0.015625,7.50287e+07)};
\addplot[color=red,mark=square*,mark size=3,semithick, mark options={fill=markercolor}]
coordinates{(0.5,361.215)(0.25,2405.62)(0.125,19806.1)(0.0625,164175)(0.03125,1.34192e+06)(0.015625,1.08595e+07)};

\logLogSlopeTriangleNeg{0.5}{0.15}{0.9}{-4}{blue};
\logLogSlopeTriangleFlipNeg{0.45}{0.15}{0.45}{-3}{red};

\node at (axis cs:.006,5e9) {$\omega = 2$};
\node at (axis cs:.006,2e8) {$\omega = 1$};
\node at (axis cs:.006,8e6) {$\omega = 1/4$};

%\legend{$L^2$ projection, LSC-DG, WADG}
\end{loglogaxis}
\end{tikzpicture}
%}
%\subfloat[$N=4$]{
%\begin{tikzpicture}
%\begin{loglogaxis}[
%legend style={font=\tiny},
%	legend cell align=left,
%	width=.45\textwidth,
%    xlabel={Mesh size $h$},
%    ylabel={$L^2$ error},
%    xmin=.0025, xmax=1,
%    ymin=1e2, ymax=1e11,
%    legend pos=south west,
%    xmajorgrids=true,
%    ymajorgrids=true,
%    grid style=dashed,
%] 
%\addplot[color=blue,mark=*,mark size=3,semithick, mark options={fill=markercolor}]
%coordinates{(0.5,12500.3)(0.25,170807)(0.125,2.58095e+06)(0.0625,4.02806e+07)(0.03125,6.3869e+08)(0.015625,1.18075e+10)};
%\addplot[color=black,mark=diamond*,mark size=4,semithick, mark options={fill=markercolor}]
%coordinates{(0.5,3264.24)(0.25,28622.5)(0.125,258261)(0.0625,2.23784e+06)(0.03125,1.87329e+07)(0.015625,1.53521e+08)};
%\addplot[color=red,mark=square*,mark size=3,semithick, mark options={fill=markercolor}]
%coordinates{(0.5,745.318)(0.25,4925.04)(0.125,40509.8)(0.0625,335834)(0.03125,2.74549e+06)(0.015625,2.22203e+07)};
%
%\logLogSlopeTriangleNeg{0.5}{0.15}{0.9}{-4}{blue};
%\logLogSlopeTriangleFlipNeg{0.45}{0.15}{0.6}{-3}{red};
%\node at (axis cs:.006,1.18e10) {$\omega = 2$};
%\node at (axis cs:.006,5e8) {$\omega = 1$};
%\node at (axis cs:.006,2e7) {$\omega = 1/4$};
%%\legend{$L^2$ projection, LSC-DG, WADG}
%\end{loglogaxis}
%\end{tikzpicture}
%}
\caption{Growth of $\max_k \nor{\frac{1}{J}}_{\LinfDk}\nor{J}_{W^{N+1,\infty}\LRp{D^k}}$ for $N=3$ curvilinear Arnold-type meshes with various warping parameters $\omega$.}
\label{fig:snorm}
\end{figure}
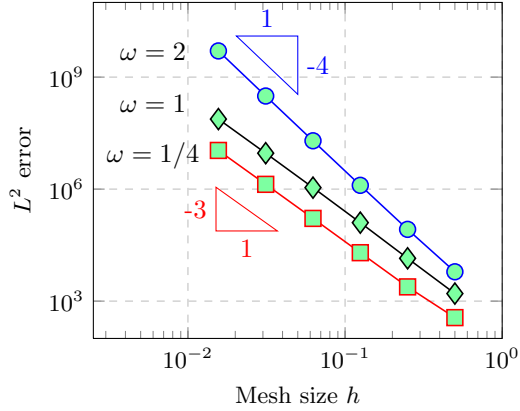

Unlike straight-sided Arnold-type meshes, determining an explicit expression for $J$ under the curved meshes used in Figure~\ref{fig:meshwarping} is more difficult.  However, it is straightforward to compute the the constant in the bound for the $L^2$ error of the WADG pseudo-projection
\[
\tilde{\kappa}_{J} \coloneqq \max_{D^k \in \Oh} \nor{\frac{1}{J}}_{\LinfDk} \nor{J}_{W^{N+1,\infty}\LRp{D^k}},
\]
where the above definition takes the maximum value of this constant over all elements.  Figure~\ref{fig:snorm} shows the growth under mesh refinement of $\tilde{\kappa}_J$ for $N=3$, which is observed to grow at a rate of $O(1/h^{N+1})$ for $\omega = 2$ and $O(1/h^N)$ for $\omega = 1,1/4$. 

Figure~\ref{fig:randarnold} shows that the asymptotic convergence rate of the WADG pseudo-projection is well-predicted by the growth of $\tilde{\kappa}_{J}$.  When $\tilde{\kappa}_J$ grows as $O(1/h^{N+1})$, the bound in Theorem~\ref{lemma:proj} is $O(1)$ irregardless of $h$, indicating that no convergence is possible.   When $\tilde{\kappa}_J$ grows as $O(1/h^{N})$, the bound in Theorem~\ref{lemma:proj} is $O(h)$, suggesting a convergence rate of at most $O(h)$.  Both predicted rates of convergence are observed in numerical experiments.  These numerical experiments also suggest that, while WADG is still sensitive to the geometric mapping, it is less sensitive than existing low-storage methods such as LSC-DG, where the $L^2$ error does not decrease under mesh refinement irregardless of the magnitude of the curvilinear warping $\omega$.  Additionally, the effect of the geometric mapping can be quantified for a given sequence of meshes by explicitly computing the constant in Theorem~\ref{lemma:proj}, as is done in Figure~\ref{fig:snorm}.  

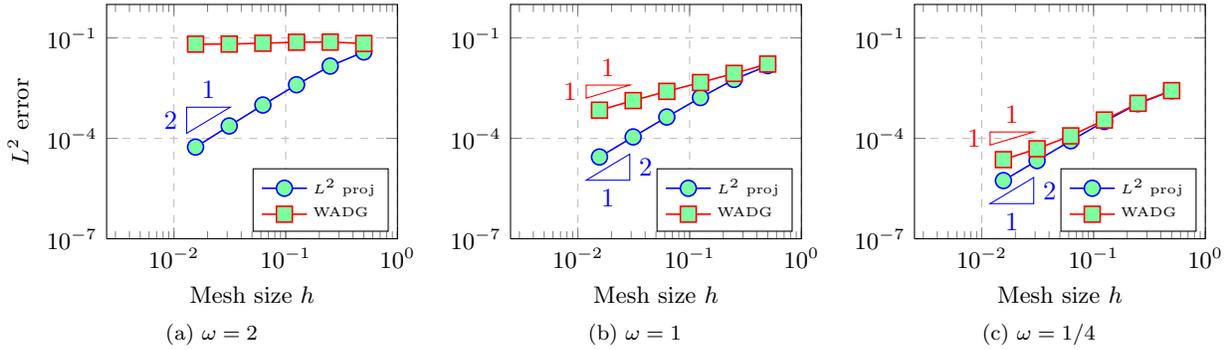
\begin{figure}
%\centering
\hspace{-1em}
\subfloat[$\omega = 2$]{
\begin{tikzpicture}
\begin{loglogaxis}[
legend style={font=\tiny},
	legend cell align=left,
	width=.33\textwidth,
    xlabel={Mesh size $h$},
    ylabel={$L^2$ error},
    xmin=.0025, xmax=1,
    ymin=1e-7, ymax=1e-0,
    legend pos=south east,
    xmajorgrids=true,
    ymajorgrids=true,
    grid style=dashed,
] 
\addplot[color=blue,mark=*,mark size=3,semithick, mark options={fill=markercolor}]
coordinates{(0.5,0.038232)(0.25,0.0144542)(0.125,0.00400233)(0.0625,0.000991275)(0.03125,0.000234827)(0.015625,5.4896e-05)};
%\addplot[color=black,mark=diamond*,mark size=4,semithick, mark options={fill=markercolor}]
%coordinates{(0.5,0.0579895)(0.25,0.052133)(0.125,0.0462352)(0.0625,0.0434683)(0.03125,0.0394193)(0.015625,0.032696)};
\addplot[color=red,mark=square*,mark size=3,semithick, mark options={fill=markercolor}]
coordinates{(0.5,0.0681969)(0.25,0.0751338)(0.125,0.0744307)(0.0625,0.0695447)(0.03125,0.0656228)(0.015625,0.0653247)};
%\logLogSlopeTriangleFlip{0.3}{0.125}{0.475}{1}{red};
\logLogSlopeTriangleFlip{0.425}{0.15}{0.45}{2}{blue};

%\node at (axis cs:.03,2.1e-07) {$a = 10^{-1}$};

%\legend{$L^2$ proj, LSC-DG, WADG}
\legend{$L^2$ proj, WADG}
\end{loglogaxis}
\end{tikzpicture}
}
\subfloat[$\omega = 1$]{
\begin{tikzpicture}
\begin{loglogaxis}[
legend style={font=\tiny},
	legend cell align=left,
	width=.33\textwidth,
    xlabel={Mesh size $h$},
    xmin=.0025, xmax=1,
    ymin=1e-7, ymax=1e-0,
    legend pos=south east,
    xmajorgrids=true,
    ymajorgrids=true,
    grid style=dashed,
] 
% adding this 2x because pgfplots dashes the line for some reason...
\addplot[color=blue,mark=*,mark size=3,semithick, mark options={fill=markercolor}]
coordinates{(0.5,0.0149174)(0.25,0.00576658)(0.125,0.00165616)(0.0625,0.000433839)(0.03125,0.000110389)(0.015625,2.78024e-05)};
%\addplot[color=black,mark=diamond*,mark size=4,semithick, mark options={fill=markercolor}]
%coordinates{(0.5,0.018737)(0.25,0.0147278)(0.125,0.0117761)(0.0625,0.0110991)(0.03125,0.0112431)(0.015625,0.0114663)};
\addplot[color=red,mark=square*,mark size=3,semithick, mark options={fill=markercolor}]
coordinates{(0.5,0.0165783)(0.25,0.00872368)(0.125,0.00464969)(0.0625,0.00253621)(0.03125,0.00134518)(0.015625,0.000695124)};
\logLogSlopeTriangleFlip{0.41}{0.15}{0.6}{1}{red};
\logLogSlopeTriangle{0.41}{0.15}{0.25}{2}{blue};

%\node at (axis cs:.03,2.1e-07) {$a = 10^{-1}$};

%\legend{$L^2$ proj, LSC-DG, WADG}
\legend{$L^2$ proj, WADG}
\end{loglogaxis}
\end{tikzpicture}
}
\subfloat[$\omega = 1/4$]{
\begin{tikzpicture}
\begin{loglogaxis}[
legend style={font=\tiny},
	legend cell align=left,
	width=.33\textwidth,
    xlabel={Mesh size $h$},
    xmin=.0025, xmax=1,
    ymin=1e-7, ymax=1e-0,
    legend pos=south east,
    xmajorgrids=true,
    ymajorgrids=true,
    grid style=dashed,
] 
% adding this 2x because pgfplots dashes the line for some reason...
\addplot[color=blue,mark=*,mark size=3,semithick, mark options={fill=markercolor}]
coordinates{(0.5,0.00264879)(0.25,0.0010862)(0.125,0.000319333)(0.0625,8.46815e-05)(0.03125,2.16915e-05)(0.015625,5.48248e-06)};
\addplot[color=red,mark=square*,mark size=3,semithick, mark options={fill=markercolor}]
coordinates{(0.5,0.00266576)(0.25,0.00111278)(0.125,0.000352944)(0.0625,0.000118832)(0.03125,4.89236e-05)(0.015625,2.31429e-05)};
\logLogSlopeTriangleFlip{0.41}{0.15}{0.4}{1}{red};
\logLogSlopeTriangle{0.41}{0.15}{0.15}{2}{blue};

%\node at (axis cs:.03,2.1e-07) {$a = 10^{-1}$};

%\legend{$L^2$ proj, LSC-DG, WADG}
\legend{$L^2$ proj, WADG}
\end{loglogaxis}
\end{tikzpicture}
}
\caption{Errors for WADG pseudo-projection and $L^2$ projection for degree $N=3$ approximations on curvilinear Arnold-type meshes at three different warping parameters.  
%The convergence rate of the WADG pseudo-projection is well-predicted by the growth of $\nor{\frac{1}{J}}_{\LinfDk}\nor{J}_{W^{N+1,\infty}\LRp{D^k}}$ in Figure~\ref{fig:snorm}, which is $O(h^{-\LRp{N+1}})$ for $\omega = 2$ (resulting in stagnation) and $O(h^{-N})$ for $\omega = 1, 1/4$ (resulting in $O(h)$ convergence). 
In all cases, LSC-DG errors stagnate under mesh refinement and are not shown.}
\label{fig:randarnold}
\end{figure}

\subsection{DG for the two-dimensional wave equation}
\label{sec:wave2d}
We next verify that $L^2$ errors for DG on curvilinear meshes converge at appropriate rates for smooth solutions.  We compute $L^2$ errors for the acoustic wave equation over the unit circle (centered at the origin).  The solution is taken to be the rotationally symmetric pressure $p(r,t)$ given by
\[
p(r,t) = J_0(\lambda r) \cos\LRp{\lambda t},
\]
where $J_0$ is zeroth order Bessel function of the first kind, $r = \sqrt{x^2 + y^2}$, and $\lambda$ satisfies Dirichlet boundary conditions $J_0(\lambda) = 0$.  In the following numerical experiments, we take $\lambda = 5.52007811028631$ and compute errors at final time $T=1$.  In this work, the solution is evolved in time using a low-storage 4th order five-stage Runge-Kutta method \cite{carpenter1994fourth}.  

For these 2D experiments, we employ the strong formulation described in Section~\ref{sec:form}, where the quadrature is chosen such that relevant volume and surface integrands are integrated exactly and the formulation is energy stable.  A nested refinement scheme is employed to generate a sequence of meshes, and for each mesh in this sequence the circular domain is approximated using an isoparametric mapping constructed through Gordon-Hall blending of the boundary elements \cite{gordon1973construction, hesthaven2007nodal}.  The meshes and convergence results are shown in Figure~\ref{fig:err2d}, where optimal rates of convergence are observed.  

\begin{figure}
%\centering
\hspace{-1em}
\subfloat[Initial mesh]{
\includegraphics[width=.29\textwidth]{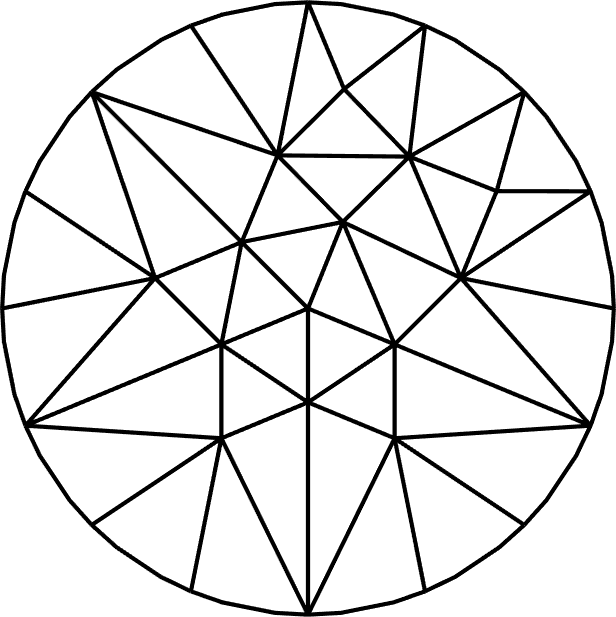}
}
\subfloat[Once-refined mesh]{
\includegraphics[width=.29\textwidth]{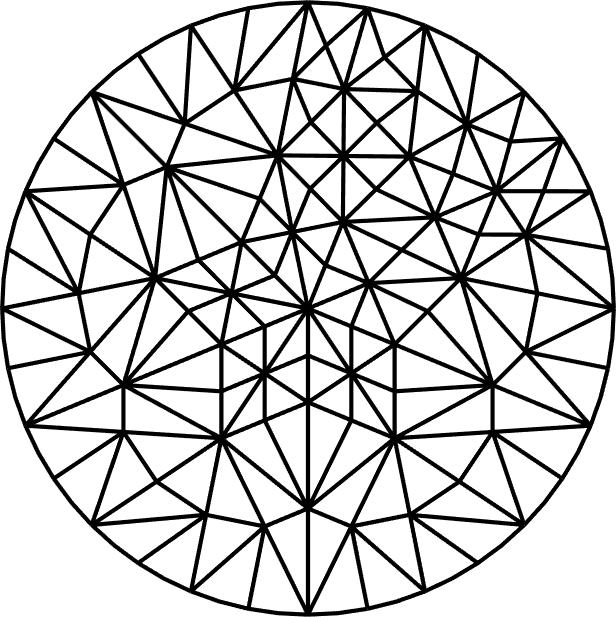}
}
\subfloat[$L^2$ errors under mesh refinement]{
\begin{tikzpicture}
\begin{loglogaxis}[
	legend cell align=left,
	width=.375\textwidth,
    xlabel={Mesh size $h$},
    ylabel={$L^2$ error},
    xmin=.025, xmax=1.1,
    ymin=5e-11, ymax=1e-1,
    legend pos=south east,
    xmajorgrids=true,
    ymajorgrids=true,
    grid style=dashed,
] 
\addplot[color=blue,mark=*,mark size=3,semithick, mark options={fill=markercolor}]
coordinates{(0.320326,0.00791816)(0.160163,0.000348412)(0.0800816,2.3037e-05)(0.0414696,1.36613e-06)};
\addplot[color=blue,mark=otimes*,mark size=3,semithick, mark options={fill=markercolor}]
coordinates{(0.320326,0.00791808)(0.160163,0.000348412)(0.0800816,2.3037e-05)(0.0414696,1.36613e-06)};
\node at (axis cs:.6,.01) {$N=3$};
\logLogSlopeTriangleFlip{0.25}{0.125}{0.52}{4}{black};

\addplot[color=red,mark=*,mark size=3,semithick, mark options={fill=markercolor}]
coordinates{(0.320326,0.000620603)(0.160163,3.0709e-05)(0.0800816,7.71624e-07)(0.0414696,2.41016e-08)};
\addplot[color=red,mark=otimes*,mark size=3,semithick, mark options={fill=markercolor}]
coordinates{(0.320326,0.000620598)(0.160163,3.0709e-05)(0.0800816,7.71624e-07)(0.0414696,2.41016e-08)};
\node at (axis cs:.6,.00125) {$N=4$};
\logLogSlopeTriangleFlip{0.25}{0.1}{0.325}{5}{black};

\addplot[color=black,mark=*,mark size=3,semithick, mark options={fill=markercolor}]
coordinates{(0.320326,0.000202842)(0.160163,2.04859e-06)(0.0800816,3.47404e-08)(0.0414696,5.56032e-10)};
\addplot[color=black,mark=otimes*,mark size=3,semithick, mark options={fill=markercolor}]
coordinates{(0.320326,0.000202852)(0.160163,2.04858e-06)(0.0800816,3.47403e-08)(0.0414696,5.56032e-10)};
\node at (axis cs:.6,.0002) {$N=5$};
\logLogSlopeTriangle{0.3}{0.125}{0.125}{6}{black};

\legend{DG, WADG}
\end{loglogaxis}
\end{tikzpicture}
}
\caption{Convergence of $L^2$ errors for curvilinear DG and WADG on a circular domain, approximated using an isoparametric curvilinear mesh.}
\label{fig:err2d}
\end{figure}
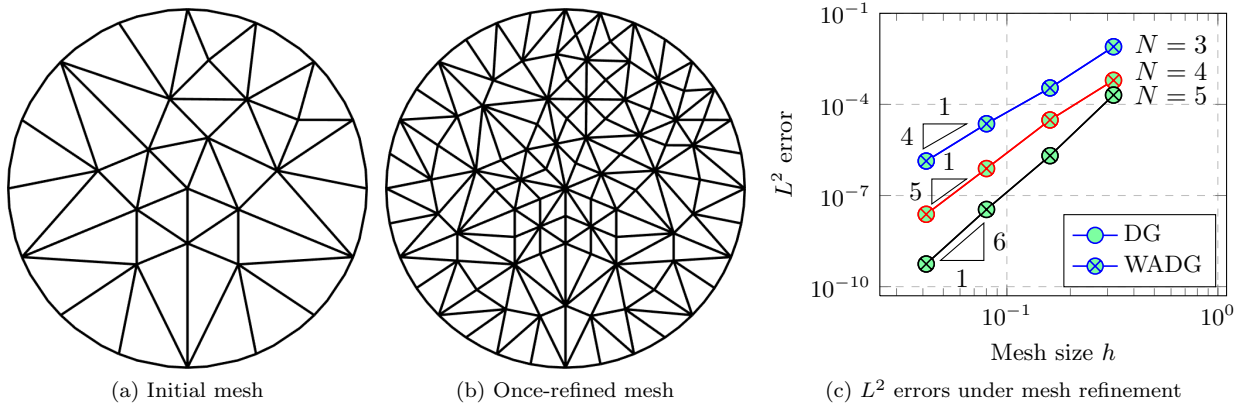

It is shown in Figure~\ref{fig:err2d} that WADG solutions are nearly identical to those obtained by curvilinear DG; however, the computational runtime and storage costs of WADG are more favorable than those of curvilinear DG.  

\subsubsection{Eigenspectra}

We also compare the spectra of the DG time evolution matrix $-\bm{M}^{-1}\bm{A}_h$.  For penalty flux parameters $\tau_p,\tau_u > 0$, there exist eigenvalues of $-\bm{M}^{-1}\bm{A}_h$ with negative real part, reflecting the dissipative nature of the equations.  The numerical flux reduces to a central flux for penalty parameters $\tau_p = \tau_u = 0$, which is non-dissipative in time and results in purely imaginary eigenvalues.  Figure~\ref{fig:eigs} shows the eigenspectra of curvilinear DG and WADG for $N=3$ overlaid on each other.  For both dissipative and non-dissipative fluxes, the difference between DG and WADG eigenvalues agrees to within two digits for large magnitude eigenvalues.  For small eigenvalues, the difference between DG and WADG eigenvalues converges rapidly.  This will be studied in future work.  

\begin{figure}
\centering
\subfloat[Spectra for $\tau = 0$]{
\includegraphics[width=.315\textwidth]{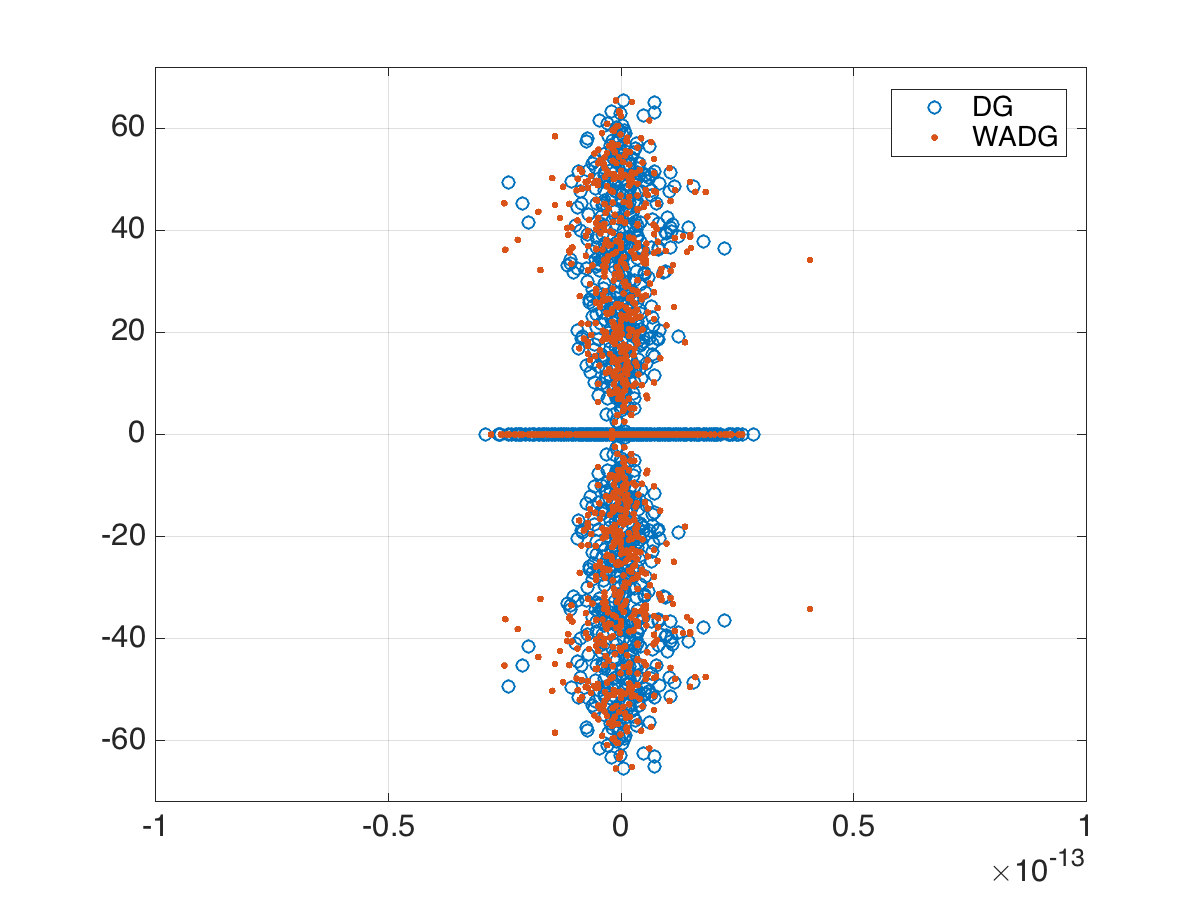}
}
\subfloat[${\rm Im}\LRp{\lambda_i}$ for $\tau = 0$]{
\includegraphics[width=.315\textwidth]{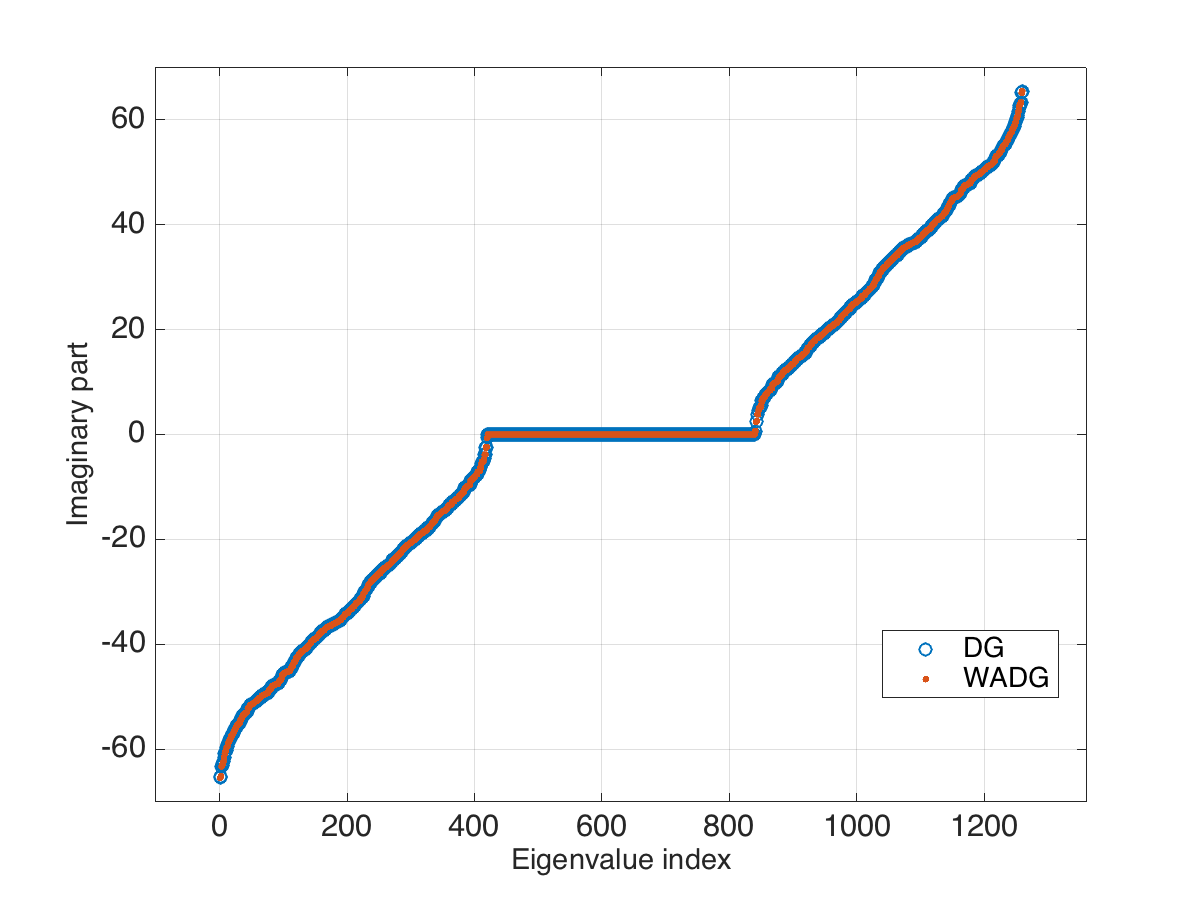}
}
\subfloat[Spectra for $\tau = 1$]{
\includegraphics[width=.315\textwidth]{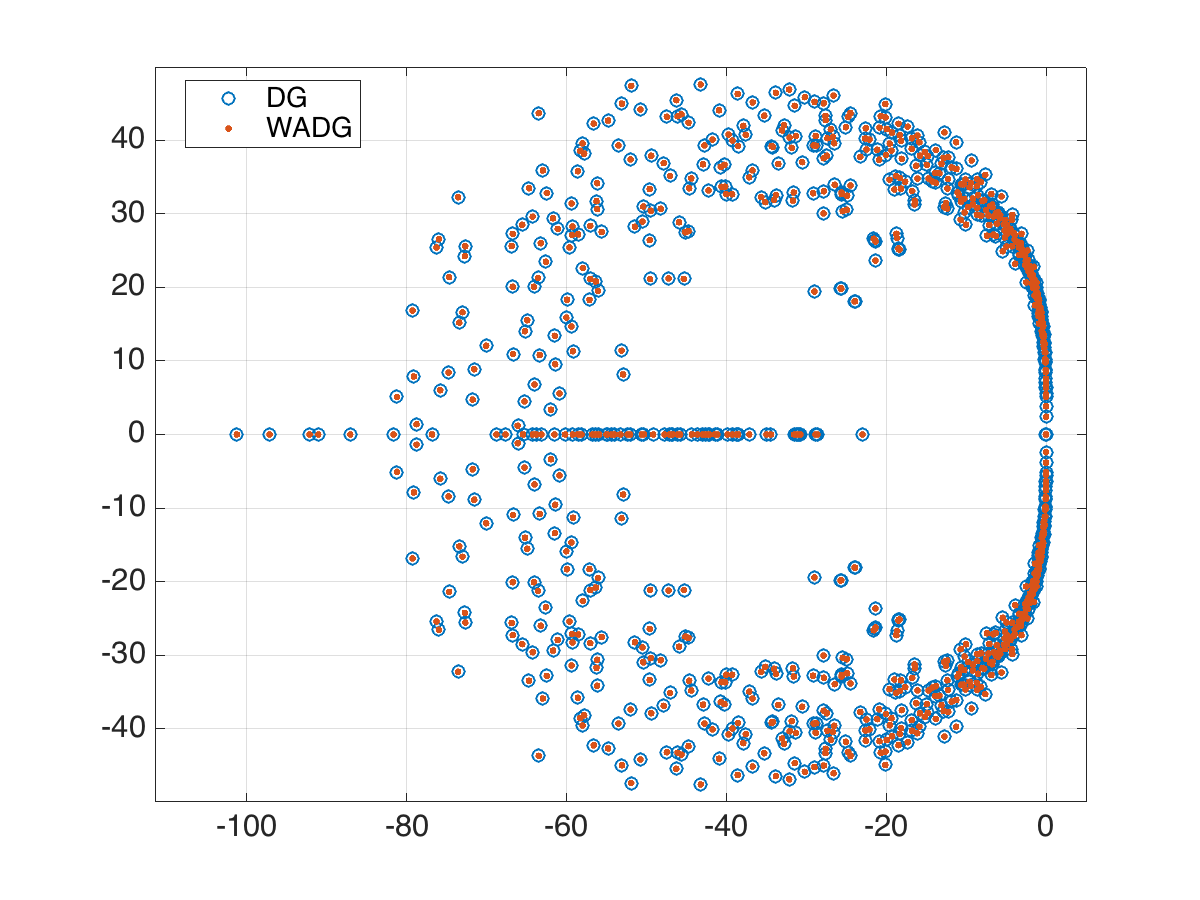}
}
\caption{Spectra of discretized wave equation using DG and WADG for $N=3$.  Eigenvalues under WADG and DG show good agreement for both central ($\tau = 0$) and dissipative ($\tau = 1$) numerical fluxes are shown.}
\label{fig:eigs}
\end{figure}

The results in Section~\ref{sec:wave2d} and Figure~\ref{fig:err2d} show that curvilinear DG and WADG result in very similar $L^2$ errors for a specific solution.  The similarity of the spectra of the curvilinear DG and WADG evolution matrices imply that both schemes possess similar numerical dispersion and dissipation as well.  

\subsection{DG for the three-dimensional wave equation}

Finally, we examine convergence of WADG on three-dimensional curved domains.  We consider the unit sphere (centered at the origin) and the spherically symmetric pressure solution $p(r,t)$ given by
\[
p(r,t) = \frac{\sin\LRp{\pi r}}{\pi r}\cos\LRp{\pi t}
\]
where $r = \sqrt{x^2 + y^2 + z^2}$.  
\begin{figure}
\centering
\subfloat[Planar initial mesh]{
\includegraphics[width=.325\textwidth]{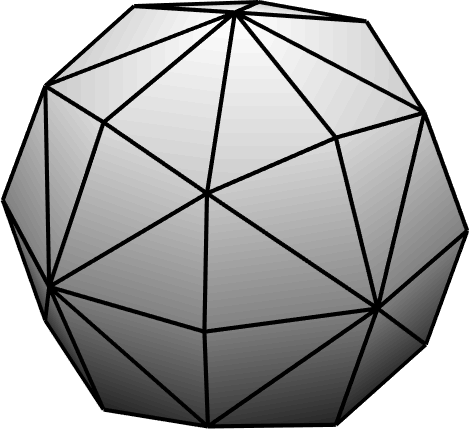}
}
%\hspace{.01\textwidth}
\subfloat[Curved initial mesh]{
\includegraphics[width=.3\textwidth]{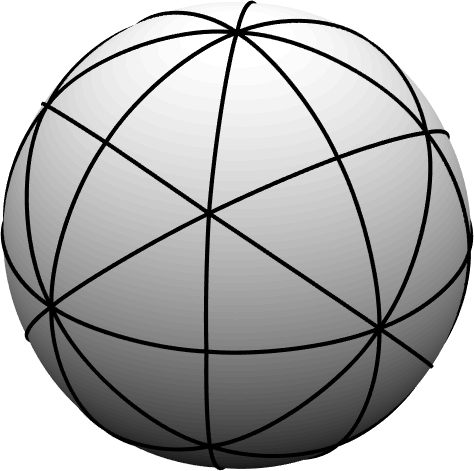}
}
%\\
%\subfloat[Planar mesh]{
%\includegraphics[width=.375\textwidth]{figs/spherePlanar2.png}
%}
%\hspace{.01\textwidth}
\subfloat[Curved refined mesh]{
\includegraphics[width=.3\textwidth]{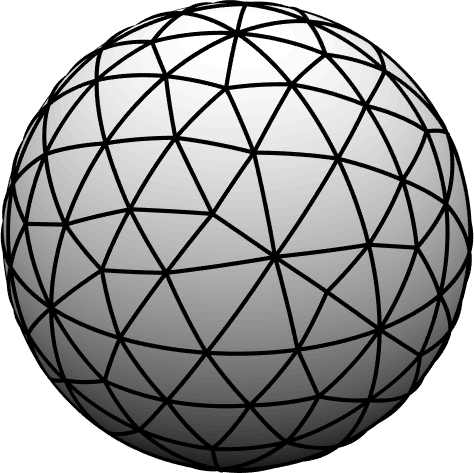}
}
\caption{Planar meshes and the resulting curvilinear meshes produced by Gordon-Hall blending.}  
\label{fig:gh3d}
\end{figure}
A sequence of non-nested meshes are constructed using \textsc{Gmsh}, and on each mesh, Gordon-Hall blending is used to construct smoothed geometric mappings through extensions of edge and face interpolations of the sphere surface, as shown in Figure~\ref{fig:gh3d}.  $L^2$ errors for the strong-weak formulation at time $t=1/4$ are computed and shown in Figure~\ref{fig:err3d}. In all cases, we observe that errors for the strong-weak formulation are virtually identical to errors for the strong form, showing disagreement only in the third significant digit and only on the coarsest mesh.   $L^2$ errors are observed to converge between the optimal rate of $O(h^{N+1})$ and the theoretical rate for DG $O(h^{N+1/2})$ \cite{warburton2013low,chan2016weight}.

\begin{figure}
\centering
\begin{tikzpicture}
\begin{loglogaxis}[
	legend cell align=left,
	width=.5\textwidth,
    xlabel={Mesh size $h$},
    ylabel={$L^2$ error},
%    xmin=400, xmax=2e6,
    xmin=5e-2, xmax=1.25,
    ymin=1e-9, ymax=.1,        
    legend pos=south east,
    xmajorgrids=true,
    ymajorgrids=true,
    grid style=dashed,
    cycle list name=color list]
] 
\addplot+[color=blue,mark=*,mark options={fill=markercolor},semithick]
coordinates{(1,0.0138761)(0.499567,0.00549824)(0.1682,9.03376e-05)(0.0830612,7.85662e-06)};

\addplot+[color=red,mark=square*,mark options={fill=markercolor},semithick]
coordinates{(1,0.00407636)(0.499567,0.000325115)(0.1682,3.81903e-06)(0.0830612,2.07829e-07)};

\addplot+[color=black,mark=triangle*,mark options={fill=markercolor},semithick]
coordinates{(1,0.000836237)(0.499567,4.83058e-05)(0.1682,1.00745e-07)};

\addplot+[color=magenta,mark=diamond*,mark options={fill=markercolor},semithick]
coordinates{(1,0.000232785)(0.499567,3.13379e-06)(0.1682,7.82233e-09)};

\logLogSlopeTriangleFlip{0.3}{0.15}{0.52}{3}{blue};
\logLogSlopeTriangleFlip{0.3}{0.15}{0.32}{4}{red};
\logLogSlopeTriangleFlip{0.475}{0.125}{0.26}{5}{black};
\logLogSlopeTriangle{0.55}{0.15}{0.1}{5.5}{magenta};

%\logLogSlopeTriangleNeg{0.85}{0.15}{0.55}{-1}{blue};
%\logLogSlopeTriangleNeg{0.85}{0.15}{0.39}{-4/3}{red};
%\logLogSlopeTriangleFlipNeg{0.75}{0.15}{0.25}{-5/3}{black};
%\addplot+[color=cyan,mark=diamond*,mark options={fill=markercolor},semithick]
%coordinates{(2688,0.00017551)(21560,1.97403e-06)(564872,9.37127e-09)};

\legend{N=2,N=3, N=4,N=5}
\end{loglogaxis}
\end{tikzpicture}

\caption{Convergence of the $L^2$ error for the WADG strong-weak formulation on a curved sphere.}
\label{fig:err3d}
\end{figure}
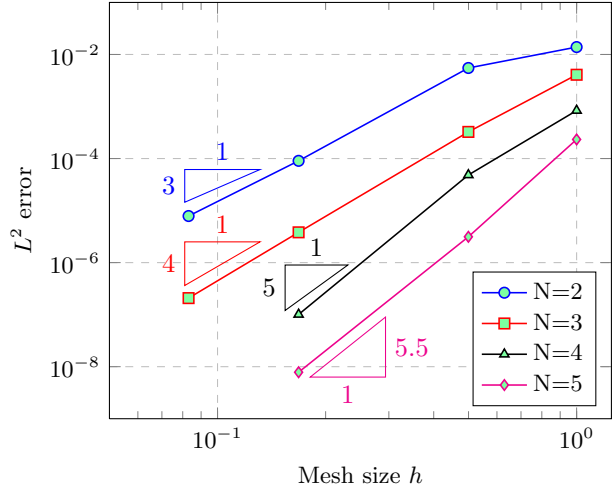

\section{Computational results}
\label{sec:comp}
In this section, we describe an implemention of WADG on Graphics Processing Units (GPUs), along with general optimization strategies to improve performance on many-core architectures.  Computational results quantifying the performance of both the strong and strong-weak WADG formulations are also reported.  

While any basis can be used with WADG, we use a nodal discontinuous Galerkin method in this work, with interpolation points placed at ``Warp and Blend'' nodes \cite{warburton2006explicit}.  Under the structure of a nodal basis, the computation of numerical fluxes and surface contributions from the DG formulation is slightly simplified, though the same structure is achievable under a Bernstein-Bezier basis as well \cite{chan2015bbdg}.  We note that the optimization strategies we describe are intended to be applied only to curvilinear elements.  Because planar elements generally admit a more efficient implementation than curvilinear elements, separate kernels should be written for planar and curvilinear elements in order to reduce runtimes where possible.   Additionally, we note that computational strategies are chosen to target low-to-moderate orders ($N\leq 5$) in this work.  For higher orders, different computational strategies may be necessary.

\subsection{GPU implementation} 

The main steps involved in a time-explicit DG solver are computation of the right hand side (evaluation of the DG formulation) and update of the solution.  The GPU implementation follows \cite{klockner2009nodal}, and is broken up into three main kernels:
\begin{itemize}
\item A \textbf{volume kernel}, which compute contributions from volume integrals within the DG variational form and stores them in global memory.
\item A \textbf{surface kernel}, which computes numerical fluxes and accumulate right hand side contributions from surface integrals within the DG variational form.
\item An \textbf{update kernel}, which applies the weight-adjustment to the DG right hand side and evolves the solution in time.  
\end{itemize}
For the strong-weak formulation, the surface kernel is broken into two smaller kernels to improve performance.  The first kernel interpolates values of the solution to face quadrature points and stores them in global memory, while the latter kernel uses these values and computes surface integral contributions to the strong-weak formulation.  

GPUs possess thousands of computational cores or processing elements, organized into synchronized workgroups.  Because the output for each degree of freedom can be computed independently, implementations of time-explicit DG methods typically assign each degree of freedom to a core, and assign one or more elements to a workgroup.  It was observed in \cite{klockner2009nodal} that assigning only a single element to a workgroup resulted in suboptimal performance at low orders of approximation.  This was remedied by processing multiple elements per-workgroup, and tuning the number of elements processed to maximize computational performance.  The number of elements processed per-workgroup for volume, surface, and update kernels is referred to as the block size.  For all computational results reported, block sizes are optimized to minimize runtime.

\subsection{Update kernel}
\label{sec:update}
In our implementation, the update kernel is the same for both the strong and weak formulations.  We use quadrature which is exact for polynomials of degree $2N+1$, with quadrature weights $\bm{w}_q$.  Let $\bm{V}_q$ be the generalized Vandermonde matrix such that
\[
\LRp{\bm{V}_q}_{ij} = \ell_j(\widehat{\bm{x}}_i),
\]
where $\ell_j$ is the $j$th nodal basis function and $\widehat{\bm{x}}_i$ is the $i$th quadrature point.  
Let $\bm{A}^k_h\bm{Q}$ denote the local right-hand-side contribution from the DG variational formulation; then, over each element $D^k$, the update kernel computes the application of the weight-adjusted mass matrix inverse
\begin{equation}
\LRp{\tilde{\bm{M}}^k}^{-1} \bm{A}^k_h\bm{Q} = \widehat{\bm{M}}^{-1}\bm{M}_{1/J^k}\widehat{\bm{M}}^{-1} \bm{A}^k_h\bm{Q}.
\label{eq:prod}
\end{equation}
We define the matrix $\bm{P}_q$ as
\[
\bm{P}_q \coloneqq \widehat{\bm{M}}^{-1} \bm{V}_q^T \diag{\bm{w}_q},
\]
where $\bm{w}_q$ are quadrature weights.  Then, using only pointwise data and reference matrices $\bm{V}_q$ and $\bm{P}_q$, (\ref{eq:prod}) may be rewritten as 
\begin{equation}
\LRp{\tilde{\bm{M}}^k}^{-1} \bm{A}^k_h\bm{Q} = \bm{P}_q \diag{1 / {\bm{J}^k_q}} \bm{V}_q \widehat{\bm{M}}^{-1}\bm{A}^k_h\bm{Q},
\label{eq:matfree}
\end{equation}
where $\diag{1/\bm{J}^k_q}$ denotes the a diagonal matrix whose entries are the evaluation of $1/J^k$ at quadrature points.  We note that (\ref{eq:matfree}) can be applied in a matrix-free fashion.  
Furthermore, $\widehat{\bm{M}}^{-1}$ is fused into reference matrices used within the volume and surface kernels to reduce the cost of evaluating $\widehat{\bm{M}}^{-1}\bm{A}^k_h$.  

The above procedure involves two matrix-vector products ($\bm{V}_q$ and $\bm{P}_q$) which are performed within the same kernel.  The output of the first matrix vector product must be stored in shared memory, requiring a shared memory array of size $N_q$, where $N_q$ is the number of quadrature points for a degree $2N+1$ quadrature.  

\subsection{Strong formulation}

In this section, we describe the implementation of the volume and surface kernels for the strong formulation
\begin{align*}
\int_{D^k}\frac{1}{c^2}\pd{p}{t}{} v &= -\int_{D^k} \Div \bm{u} v - \int_{\partial D^k} \frac{1}{2}\LRp{\jump{\bm{u}}\cdot \bm{n}^- - \tau_p\jump{p}} v \\
\int_{D^k}\pd{\bm{u}}{t}{} \cdot \bm{\tau} &= - \int_{D^k} \Grad p \cdot \bm{\tau} - \int_{\partial D^k} \frac{1}{2}\LRp{ \jump{p} - \tau_u\jump{\bm{u}}\cdot\bm{n}^-}\bm{\tau}\cdot \bm{n}^-.
\end{align*}
Computation of the above integrals computes the action of the matrix $\bm{A}^k_h$ on a vector.  Because the implementation of the WADG update kernel in Section~\ref{sec:update} requires $\widehat{\bm{M}}^{-1}\bm{A}^k_h$, the inverse of the reference mass matrix will also be premultiplied into reference arrays used in the volume and surface kernels.  

\paragraph{Volume kernel} The volume contributions are relatively straightforward to compute.  We will describe how to compute volume terms for the pressure equation; volume terms for the velocity equations are computed similarly.  For example, we can rewrite the volume term on the pressure equation as
\[
\int_{D^k} \Div \bm{u} v = \int_{\widehat{D}} \LRp{\pd{\bm{u}_1}{x} + \pd{\bm{u}_2}{y} + \pd{\bm{u}_3}{z}} vJ = \int_{\widehat{D}} \Pi_N\LRp{\LRp{\pd{\bm{u}_1}{x} + \pd{\bm{u}_2}{y} + \pd{\bm{u}_3}{z}}J } v.
\]
Let $\bm{U}_{\rm div}$ denote the degrees of freedom for the projection $\Pi_N \LRp{\Div \bm{u} J}$; then, the volume contributions can be represented algebraically as
\[
\widehat{\bm{M}} \bm{U}_{\rm div}.
\]
Multiplying by the inverse of the reference mass matrix $\widehat{\bm{M}}^{-1}$ (resulting from the update step) removes the multiplication by $\widehat{\bm{M}}$ involved in the volume contribution, and all that remains is to compute $\Pi_N \LRp{\Div \bm{u} J}$.  

Quadrature-based projections can be performed as follows: we introduce generalized Vandermonde and projection matrices $\bm{V}_q, \bm{P}_q$ as defined in Section~\ref{sec:update}.  For the strong formulation volume kernel, these matrices must be defined for a degree $4N-3$ quadrature to ensure energy stability.  We also introduce quadrature-based derivative matrices $\bm{V}^{\widehat{x}}_q,\bm{V}^{\widehat{y}}_q,\bm{V}^{\widehat{z}}_q$ such that
\[
\LRp{\bm{V}^{\widehat{x}}_q }_{ij} = \pd{\ell_j}{\widehat{x}}(\bm{x}_i), \qquad i = 1,\ldots,N_q,
\]
where $\ell_j$ is the $j$th basis function and $\bm{x}_i$ is the $i$th quadrature point.  $\bm{V}^{\widehat{y}}_q$ and $\bm{V}^{\widehat{z}}_q$ are defined similarly.  Let $\bm{U}^k$ denote the degrees of freedom for $\bm{u}$ on $D^k$; then, $\Div \bm{u}$ can be evaluated at quadrature points using quadrature-based derivative matrices and geometric factors evaluated at quadrature points.  $\bm{U}_{\rm div}$ can then be computed by scaling these pointwise quadrature values by $\bm{J}_q$ and multiplying by $\bm{P}_q$.  As with the update kernel, this approach requires shared memory arrays of size $N_q$, where $N_q$ is the number of points in a degree $4N-3$ quadrature rule.  We note that this implies that the strong formulation with full integration is only efficient for low $N$; for larger $N$, this results in heavy use of shared memory, which decreases occupancy and results in sub-optimal performance.  
%As with the update kernel, we partition the multiplication by $\bm{V}_q$ and $\bm{P}_q$ into blocks to reduce shared memory usage and improve computational performance as described in Section~\ref{sec:update}.  

\paragraph{Surface kernel}
The surface kernel computes the pressure and velocity surface contributions to the DG formulation
\[
\int_{\partial D^k} \frac{1}{2}\LRp{\jump{\bm{u}}\cdot \bm{n}^- - \tau_p\jump{p}} v, \qquad 
\int_{\partial D^k} \frac{1}{2}\LRp{ \jump{p} - \tau_u\jump{\bm{u}}\cdot\bm{n}^-}\bm{\tau}\cdot \bm{n}^-.
\]
We will describe the computation of the pressure equation contribution; the surface contributions to the velocity equation are computed similarly.  The integral over $\partial D^k$ can be computed through integrals over each face $f \in \widehat{D}^k$.  Mapping the triangular face to a reference triangle then gives
\[
\int_{\widehat{f}} \frac{1}{2}\LRp{\jump{\bm{u}}\cdot \bm{n}^- - \tau_p\jump{p}} v J^f,
\]
where $J^f$ is the surface Jacobian for the face $f$.  Computing these contributions requires computation of jumps of $p$ and $u$.  Here, we can take advantage of the structure of nodal DG methods: assuming that the element $D^k$ shares a face $f$ with the neighboring element $D^{k,+}$,  the numerical flux over $f$ can be computed at each node on face $f$ using nodal values from $D^k$ and $D^{k,+}$.  Computation of DG surface integral contributions may then be computed through quadrature after interpolating the numerical flux to quadrature points on $f$.  

The remainder of the surface kernel is similar to that of the volume kernel.  Aside from the computation of the numerical flux, two matrix-vector multiplications are required, one to interpolate the numerical flux to face quadrature points, and one to compute the degrees of freedom $\bm{F}_p$ of the projection of the surface integrand.  The first step is done using the Vandermonde matrix $\bm{V}^{\rm tri}_q$ for quadrature points on a triangular face 
\[
\LRp{\bm{V}^{\rm tri}_q}_{ij} = \ell_j^{\rm tri}\LRp{\bm{x}^f_i}, \qquad i = 1,\ldots,N^f_q,
\]
where $\ell_j^{\rm tri}$ is the $j$th nodal basis function on a triangle and $\bm{x}^f_i$ is the $i$th triangular face quadrature point.  $\bm{V}^{\rm tri}_q$ is a linear operator which inputs nodal values on a triangular face and outputs values at face quadrature points.  We note that the face interpolation matrix is identical for each triangular face of a tetrahedron.  

Once the numerical flux is interpolated to face quadrature points, it is scaled by $\bm{J}_q^f$ (the evaluation of the surface Jacobian $J^f$ at face quadrature points) and multiplied by the face projection matrix $\bm{P}^f_q$ to compute the WADG surface contribution, where $\bm{P}^f_q$ is defined as 
\[
\bm{P}^f_q = \widehat{\bm{M}}^{-1} \LRp{\bm{V}^f_q}^T \diag{\bm{w}^f_q}, \qquad \LRp{\bm{V}^f_q }_{ij} = \ell_j \LRp{\bm{x}^f_i}, \qquad i = 1,\ldots,N^f_q,
\]
where $\ell_j$ is the $j$th nodal basis function on the tetrahedron and $\bm{x}^f_i$, $\bm{w}^f_q$ are face quadrature points and weights, respectively.  In practice, we concatenate all tetrahedral face projection matrices into a single surface projection matrix and perform a single matrix-vector multiplication.  %Since the number of quadrature nodes for a triangular face is smaller than the number of quadrature nodes for a tetrahedron, the amount of shared memory used by the surface kernel is reasonable for $N \leq 5$ and we do not apply the blocked strategy described in Section~\ref{sec:update}.  

\subsection{Strong-weak formulation}

In this section, we detail the implementation of volume and surface kernels for the strong-weak DG formulation
\begin{align*}
\int_{D^k}\frac{1}{c^2}\pd{p}{t}{} v &= \int_{D^k} \bm{u}\cdot \Grad v - \int_{\partial D^k} \frac{1}{2}\LRp{2\avg{\bm{u}}\cdot \bm{n}^- - \tau_p\jump{p}} v \\
\int_{D^k}\pd{\bm{u}}{t}{} \cdot \bm{\tau} &= - \int_{D^k} \Grad p \cdot \bm{\tau} - \int_{\partial D^k} \frac{1}{2}\LRp{ \jump{p} - \tau_u\jump{\bm{u}}\cdot\bm{n}^-}\bm{\tau}\cdot \bm{n}^-.
\end{align*}

\paragraph{Volume kernel} 

Consider the pressure equation volume contribution for the strong-weak formulation 
\[
\int_{D^k} \bm{u}\cdot \Grad v =  \int_{\widehat{D}} \LRp{\bm{u}_1 \pd{v}{x} + \bm{u}_2 \pd{v}{y} + \bm{u}_3 \pd{v}{z}} J.
\]
This contribution becomes more involved to evaluate, due to the fact that derivatives now lie on the pressure test function $v$.  Algebraic manipulation shows that this contribution can be evaluated as 
\[
\LRp{\bm{V}^{\widehat{x}}_q}^T \bm{U}_q^{\widehat{x}} + \LRp{\bm{V}^{\widehat{y}}_q}^T \bm{U}_q^{\widehat{y}} + \LRp{\bm{V}^{\widehat{z}}_q}^T \bm{U}_q^{\widehat{z}},
%\LRp{\LRp{\bm{V}^{\widehat{x}}_q}^T \diag{\bm{x}_{\widehat{x}}} + \LRp{\bm{V}^{\widehat{y}}_q}^T \diag{\bm{x}_{\widehat{y}}}}
\]
where $\LRp{\bm{V}^{\widehat{x}}_q}^T$ are quadrature-based derivative matrices.  Due to the fact that the strong-weak formulation is \textit{a priori} stable, we use a quadrature rule of degree $2N+1$, in contrast to the quadrature rule taken for the strong formulation.  The terms $\bm{U}_q^{\widehat{x}},\bm{U}_q^{\widehat{y}}$, and $\bm{U}_q^{\widehat{z}}$ are defined at quadrature points as
\begin{align*}
\bm{U}_q^{\widehat{x}} &= \diag{\bm{J}_q} \LRp{\diag{\bm{x}_{\widehat{x}}}\bm{V}_q \bm{U}_1 +\diag{\bm{y}_{\widehat{x}}}\bm{V}_q \bm{U}_2 + \diag{\bm{z}_{\widehat{x}}} \bm{V}_q \bm{U}_3},\\
\bm{U}_q^{\widehat{y}} &= \diag{\bm{J}_q} \LRp{\diag{\bm{x}_{\widehat{y}}}\bm{V}_q \bm{U}_1 +\diag{\bm{y}_{\widehat{y}}}\bm{V}_q \bm{U}_2 + \diag{\bm{z}_{\widehat{y}}} \bm{V}_q \bm{U}_3},\\
\bm{U}_q^{\widehat{z}} &= \diag{\bm{J}_q} \LRp{\diag{\bm{x}_{\widehat{z}}}\bm{V}_q \bm{U}_1 +\diag{\bm{y}_{\widehat{z}}}\bm{V}_q \bm{U}_2 + \diag{\bm{z}_{\widehat{z}}} \bm{V}_q \bm{U}_3},
\end{align*}
where $\bm{x}_{\widehat{x}},\ldots$ are evaluations of geometric factors at quadrature points.  
Along with the gradient of $p$ at quadrature points, $\bm{U}_q^{\widehat{x}},\bm{U}_q^{\widehat{y}}$, and $\bm{U}_q^{\widehat{z}}$ need to be present in shared memory in order to apply $\bm{P}_q$ and compute volume contributions.  We loop over quadrature points to reduce shared memory usage in manner similar to that of the update kernel and strong formulation volume kernel. Shared memory usage increases slightly relative to the strong formulation volume kernel, because the strong-weak formulation requires the storage of 6 additional quadrature arrays, while the strong formulation only requires 4 additional arrays.  

In addition to the previously described optimizations, the strong-weak volume kernel packs both $\bm{V}_q$, $\bm{V}^{\widehat{x}}_q$, $\bm{V}^{\widehat{z}}_q$,  $\bm{V}^{\widehat{z}}_q$ and $\bm{P}_q$, $\bm{P}^{\widehat{x}}_q$,  $\bm{P}^{\widehat{z}}_q$,  $\bm{P}^{\widehat{z}}_q$ into single \verb+float4+ (or \verb+double4+) arrays to take advantage of fast memory accesses \cite{chan2015bbdg}. This is observed to result in a significant ($10-20\%$) speedup in runtime of the strong-weak volume kernel.  The same strategy is less effective for the strong formulation volume kernel.  Because the operators $\bm{V}^{\widehat{x}}_q,  \bm{V}^{\widehat{z}}_q,  \bm{V}^{\widehat{z}}_q$ and $\bm{P}_q$ are requested separately, storing them in the same \verb+float4+ array would require two \verb+float4+ reads instead of 4 \verb+float+ reads.  It would only be advantageous to concatenate $\bm{V}^{\widehat{x}}_q,  \bm{V}^{\widehat{z}}_q,  \bm{V}^{\widehat{z}}_q$ into a single \verb+float4+ array, but because only three matrices are required, loading these matrices using \verb+float4+ array also results in one extraneous load, losing any advantage gained by fast \verb+float4+ array accesses.  

\paragraph{Surface kernel}

The surface kernel for the strong-weak formulation computes the following contributions
\[
\int_{\partial D^k} \frac{1}{2}\LRp{2\avg{\bm{u}}\cdot \bm{n}^- - \tau_p\jump{p}} v, \qquad 
\int_{\partial D^k} \frac{1}{2}\LRp{ \jump{p} - \tau_u\jump{\bm{u}}\cdot\bm{n}^-}\bm{\tau}\cdot \bm{n}^-.
\]
We observe that breaking this kernel into two smaller kernels improves performance at higher orders.  The first kernel writes values of the field variables at face quadrature points to global memory.  The second kernel reads in these values, and for each element, retrieves values at neighboring quadrature points and evaluates the numerical flux.  The second kernel also computes strong-weak surface contributions through multiplication of the numerical fluxes by $\bm{P}^f_q$.  %This is done to reduce the amount of shared memory required: unlike the numerical fluxes for the strong formulation, the strong-weak fluxes require the computation and storage of averages of $\bm{u}$ in addition to jumps of $p,\bm{u}$, which increases shared memory storage costs.\footnote{A strong-weak formulation may also be constructed by integrating the velocity equation by parts instead of the pressure equation.  Under this formulation, the surface kernel only needs to compute four scalar jumps and the scalar average of $p$.}   

%\note{explain here that strong-weak requires both jump/avg, and thus $7+4$ fluxes instead of $4+4$ fluxes for strong form.  }
%\note{finish - add that interp to surface quad pts can be done from within this kernel, but for parallelization, may want to split off to hide latency.}
%\note{Tried blocking over faces in one single kernel but didn't get good results at higher orders.  }

We note that the retrieval of neighboring flux information from $D^{k,+}$ involves non-coalesced data accesses \cite{klockner2009nodal}, and because the number of face (triangular) quadrature points is greater than the number of face nodal points, more non-coalesced accesses are required here than when only communicating nodal values.  Writing face quadrature values to global memory also involves more work for meshes with both curved and planar elements, as it becomes necessary to write out face quadrature values for all neighbors of a curved element, regardless of whether those neighbors are curved or planar.  Transferring nodal data and interpolating locally to surface quadrature (as is done in the strong formulation) avoids this extra step.  

\subsection{Comparison of computational performance}

In this section, we compare the computational performance of kernels for both strong and strong-weak formulations.  Experiments were performed using a Nvidia 980 GTX GPU.  

\subsubsection{Kernel runtimes}

\begin{figure}
\centering
\subfloat[Total runtime]{
\begin{tikzpicture}
\begin{semilogyaxis}[
	legend cell align=left,
	width=.45\textwidth,
	xlabel={Degree $N$},
	xmin=1.5, xmax=5.5,
	xtick={2,3,4,5},
    ylabel={Runtime per-degree of freedom (s)},
    ymin=5e-11, ymax=2e-8,        
    legend pos=south west,
    xmajorgrids=true,
    ymajorgrids=true,
    grid style=dashed,
    cycle list name=color list]
] 
\addplot+[color=blue,mark=*,mark options={fill=markercolor},semithick]
coordinates{(2,6.52e-10)(3,1.2179e-09)(4,2.2906e-09)(5,1.73002e-08)};

\addplot+[color=red,mark=square*,mark options={fill=markercolor},semithick]
%coordinates{(2,8.776e-10)(3,1.1444e-09)(4,1.19324e-09)(5,1.60326e-09)};
coordinates{(2,8.413e-10)(3,1.0929e-09)(4,1.11184e-09)(5,1.47466e-09)};

%\addplot+[color=blue,mark=*,mark options={solid,fill=markercolor},dashed,semithick]
%coordinates{(2,5.929e-10)(3,9.225e-10)(4,1.1375e-09)(5,1.2314e-09)};

%\legend{Strong, Strong-weak, Strong (underintegrated)}
\legend{Strong, Strong-weak}
\end{semilogyaxis}
\end{tikzpicture}
}
\subfloat[Kernel runtimes]{
\begin{tikzpicture}
\begin{semilogyaxis}[
	legend cell align=left,
	legend style={font=\tiny},	
	width=.45\textwidth,
	xlabel={Degree $N$},
	xmin=1.5, xmax=5.5,
	xtick={2,3,4,5},
    ylabel={Runtime per-degree of freedom (s)},
    ymin=5e-11, ymax=2e-8,        
    legend pos=north west,
    xmajorgrids=true,
    ymajorgrids=true,
    grid style=dashed,
    cycle list name=color list]
] 
\addplot+[color=blue,mark=*,mark options={fill=markercolor},semithick]
coordinates{(2,1.588e-10)(3,5.32e-10)(4,1.313e-09)(5,1.483e-08)};
\addplot+[color=blue,mark=*,mark options={solid,fill=markercolor},semithick,dashed]
coordinates{(2,3.223e-10)(3,4.003e-10)(4,7.079e-10)(5,2.105e-09)};

\addplot+[color=red,mark=square*,mark options={fill=markercolor},semithick]
%coordinates{(2,2.117e-10)(3,3.963e-10)(4,5.47e-10)(5,8.897e-10)};
coordinates{(2,1.754e-10)(3,3.448e-10)(4,4.656e-10)(5,7.611e-10)};

\addplot+[color=black,mark=square*,mark options={solid,fill=markercolor},semithick,dashed]
coordinates{(2,1.652e-10)(3,1.285e-10)(4,9.354e-11)(5,9.656e-11)};
\addplot+[color=red,mark=square*,mark options={solid,fill=markercolor},semithick,dashed]
coordinates{(2,3.298e-10)(3,3.34e-10)(4,2.83e-10)(5,2.518e-10)};

\legend{Volume (Strong), Surface (Strong), Volume (Strong-weak), Surface 1 (Strong-weak), Surface 2 (Strong-weak)}
\end{semilogyaxis}
\end{tikzpicture}
}
\caption{Runtimes per-degree of freedom for energy stable strong and strong-weak formulation kernels at various polynomial degrees $N$.}
\label{fig:runtimes}
\end{figure}
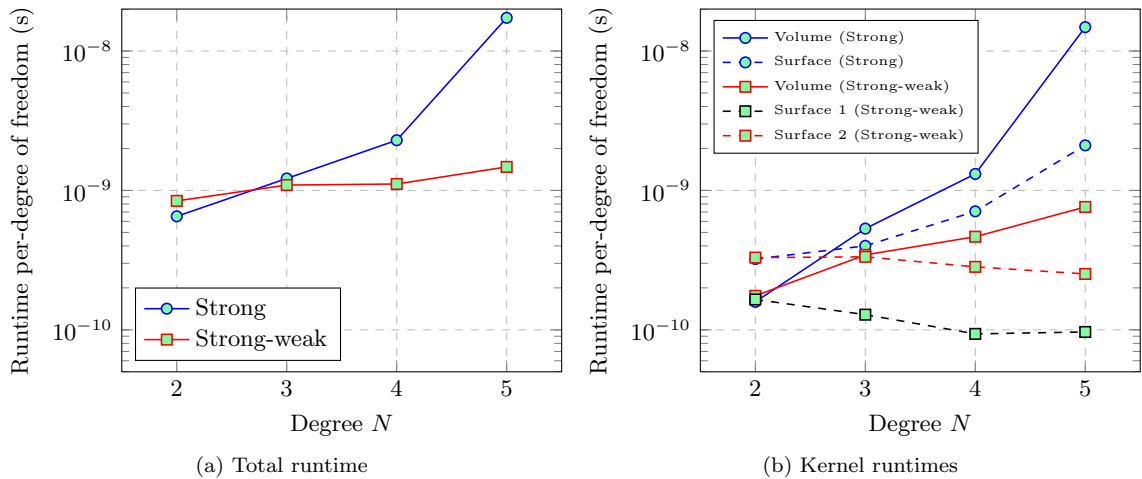

Figure~\ref{fig:runtimes} shows both total runtimes and individual kernel runtimes (reported per-degree of freedom) for both the strong and strong-weak formulations.  For both volume kernels, the runtime per-degree of freedom increases rapidly as $N$ increases.  The runtime for the strong formulation volume kernel increases far more rapidly due to the use of degree $4N-3$ quadrature rules.  At higher orders, the strong formulation volume kernel is also less efficient due to the high number of quadrature points and large amount of shared memory required, which reduces occupancy.\footnote{When pairing the strong formulation with an underintegrated quadrature rule of degree $(2N+1)$ (identical to the strength of quadrature used for the strong-weak formulation), the underintegrated strong form is faster than the strong-weak form for $N \leq 5$.  Numerical experiments suggest that decreasing the strength of quadrature for the strong formulation still results in a discrete formulation which is energy stable (for all tested meshes), and numerical errors for the underintegrated strong formulation are nearly identical to those with full quadrature.  However, energy stability is not theoretically guaranteed for the underintegrated strong formulation.  }
%\footnote{The number of quadrature points also increases disproportionately when the order grows from $N=4$ to $N=5$.  This is due to the fact that the quadrature rules employed \cite{xiao2010quadrature} are only constructed up to degree $15$ on tetrahedra, while the strong formulation requests a degree $4N-3 = 18$ quadrature rule.  This rule is constructed instead through a tensor product of one-dimensional Gauss quadratures \cite{karniadakis2013spectral}, and results in a much larger number of quadrature points. } 

The computation of the strong-weak surface contribution is broken into two separate kernels, the first of which computes and stores (in global memory) the values of solutions at quadrature points on faces. The second kernel reads these values in, computes surface integral contributions using quadrature, and accumulates these contributions in global memory.  Growth in runtime per-degree of freedom as $N$ increases is observed for the strong formulation surface kernel, though the per-degree of freedom runtimes of the strong-weak surface kernel decrease with $N$ up to $N = 5$.  Results in \cite{modave2016gpu} indicate that this is pre-asymptotic behavior, and that costs again increase for $N$ sufficiently large.  

A similar growth in $N$ of the per-degree of freedom runtimes for both volume and surface kernels is also observed for planar tetrahedra at sufficiently high orders \cite{modave2016gpu}.  Per-degree of freedom runtimes for the strong-weak surface kernels are also observed to decrease with $N$  for small $N$; however, asymptotic costs indicate that as the order is increased further, per-degree of freedom runtimes will increase in a similar manner for the volume and surface kernels of both formulations.  Several options exist to reduce this growth in runtime for higher $N$.  Modave et al.\ utilized optimized dense linear algebra libraries such as CUBLAS \cite{modave2016gpu}, which showed significant decreases in runtime for $N > 7$.  An element-per-thread data layout combined with blocking of dense elemental matrices also yields more efficient behavior at larger $N$ without requiring additional global memory \cite{chan2015bbdg}, though this approach results in slightly larger runtimes than using CUBLAS.  An alternative approach to reducing runtime for high $N$ was taken in \cite{chan2015bbdg}, where a change of polynomial basis resulted in sparse matrices and a roughly constant runtime per-degree of freedom.  However, this strategy is currently restricted to non-curved tetrahedral elements.  

%\note{Modave performance analysis - cost per-degree of freedom decreases, but will increase as order increases \cite{modave2016gpu}.  Need specialized techniques - CUBLAS or BB basis \cite{chan2015bbdg} - at higher orders of approximation.}

\subsubsection{Profiled GFLOPS and bandwidth}

\pgfplotstableread[col sep=space]{
N V S U
2                   150.237                   147.311                   137.614
3                   73.1056                   121.633                    72.458
4                   38.0127                   66.6634                    84.952
5                  12.60117                     22.47                    61.802
}\BWStrong

\pgfplotstableread[col sep=space]{
N V S U
2 790.931989924433 528.699968973006 367.173785839672          
3 907.236842105263 791.906070447165 458.464635854342          
4 918.920683277119 767.340019776805 871.470946554373          
5 550.348653659875 389.616559212759 1060.02337271163
}\GFLOPSStrong

Due to the high degree of quadrature required for energy stability, the strong formulation is more expensive at any order $N > 2$.  For this reason, we focus in this section on performance results for the strong-weak formulation.  Figure~\ref{fig:strongweak} shows profiled GFLOPS per-second and bandwidth for the strong-weak formulation kernels.  As $N$ increases, the bandwidth decrease, which is observed for nodal DG as well \cite{modave2016gpu, chan2015bbdg}.  In contrast to other kernels, the profiled computational performance of the strong-weak surface kernel remains relatively high for larger $N$.  This is likely due to the splitting of the strong-weak surface kernel into to two kernels, which more efficiently utilize GPU bandwidth.  We note that adopting this split kernel strategy for the strong-weak volume kernel (as is done in \cite{chan2015gpu}) and update kernel may improve runtime, GFLOPS, and bandwidth at higher $N$, though for $N \leq 5$, this split kernel approach was observed to result in a greater overall runtime than using monolithic volume and update kernels.  

\pgfplotstableread[col sep=space]{
N V S F 
                         2                   135.172                   159.619                    99.098
                         3                    74.142                   154.132                    137.86
                         4                   52.1763                   132.888                   136.989
                         5                   32.8774                   135.096                   125.908
}\BWSkew

\pgfplotstableread[col sep=space]{
N V S F
                         2          562.428734321551          227.171249952398          203.389830508475
                         3          586.173143851508          419.977795875644           466.92607003891
                         4          770.787309769269           588.95225880893           696.41711719967
                         5          771.961869099236          943.268688966019          869.925434962717
 }\GFLOPSSkew

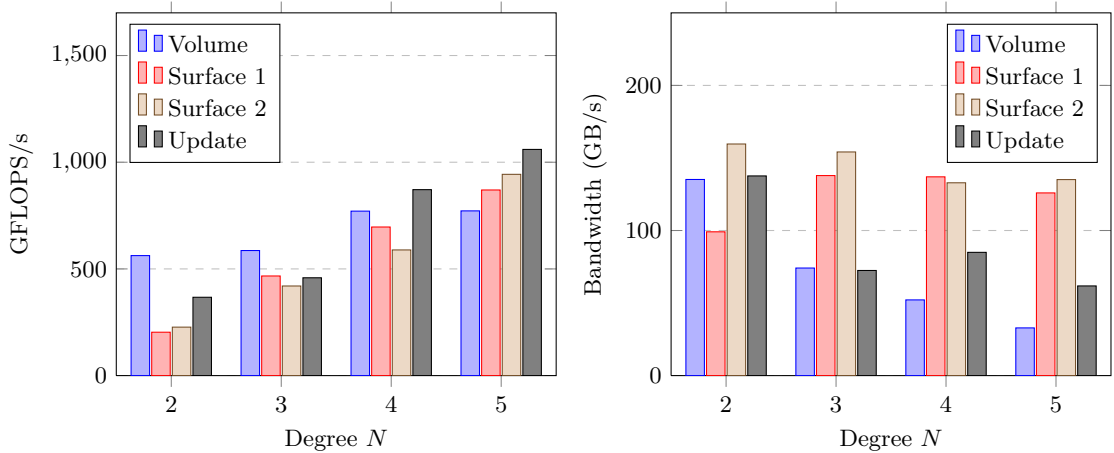
\begin{figure}
\centering
\subfloat{
\begin{tikzpicture}
\begin{axis}[
	width=.45\textwidth,
	legend cell align=left,
%	legend style={font=\tiny},
%	legend image post style={scale=0.5},
	xlabel={Degree $N$},
	ylabel={GFLOPS/s},
	xmin=1.5, xmax=5.5,
	ymin=0,ymax=1700,
        ybar=2*\pgflinewidth,
    bar width=7pt,
	xtick={2,3,4,5},
	legend pos=north west,
%	xmajorgrids=true,
	ymajorgrids=true,
	grid style=dashed,
] 
\addplot table[x=N, y=V] from \GFLOPSSkew;
\addplot table[x=N, y=F] from \GFLOPSSkew;
\addplot table[x=N, y=S] from \GFLOPSSkew;
\addplot table[x=N, y=U] from \GFLOPSStrong;
\legend{Volume, Surface 1, Surface 2,Update}
\end{axis}
\end{tikzpicture}
}
%\hspace{1em}
\subfloat{
\begin{tikzpicture}
\begin{axis}[
	width=.45\textwidth,
	legend cell align=left,
%	legend style={font=\tiny},
%	legend image post style={scale=0.5},
	xlabel={Degree $N$},
	ylabel={Bandwidth (GB/s)},
	xmin=1.5, xmax=5.5,
	ymin=0,ymax=250,
             ybar=2*\pgflinewidth,
             bar width=7pt,
	xtick={2,3,4,5},
	legend pos=north east,
%	xmajorgrids=true,
	ymajorgrids=true,
	grid style=dashed,
] 
\addplot table[x=N, y=V] from \BWSkew;
\addplot table[x=N, y=F] from \BWSkew;
\addplot table[x=N, y=S] from \BWSkew;
\addplot table[x=N, y=U] from \BWStrong;
\legend{Volume, Surface 1, Surface 2, Update}
\end{axis}
\end{tikzpicture}
}
\caption{Profiled GFLOPS and bandwidth for strong-weak DG formulation volume and surface kernels.  Results are presented for an Nvidia GTX 980 GPU, on a spherical mesh of 49748 elements.  }%The surface kernel is split into two parts, the first of which computes values at quadrature points and writes to global memory, and the second of which reads these values from global memory and computes integrals using quadrature.   }
\label{fig:strongweak}
\end{figure}

\section{Heterogeneous media in curved domains}
\label{sec:het}
A convenient aspect of WADG is that it is possible to incorporate the modeling of heterogeneous media on curvilinear elements at no additional computational cost.  Assuming now that $c^2$ varies spatially, this results in a different weighting of the mass matrix for the pressure equation
\[
\LRp{\bm{M}_{J^k/c^2} }_{ij} = \int_{\widehat{D}} \ell_j\ell_i \frac{J}{c^2}.  
\]
The inverse of this mass matrix can again be approximated using a weight-adjusted mass matrix
\[
\LRp{\bm{M}_{J^k/c^2} }^{-1} \approx \widehat{\bm{M}}^{-1} \bm{M}_{c^2/J^k} \widehat{\bm{M}}^{-1}. 
\]
Recall that, for curvilinear elements in isotropic media, the update kernel computes the following matrix-vector product
\[
\bm{P}_q {\rm diag}\LRp{1 / {\bm{J}^k_q}}\bm{V}_q \widehat{\bm{M}}^{-1}\bm{A}^k_h\bm{Q},
\]
where $\bm{A}^k_h\bm{Q}$ is the contribution from the DG variational formulation.  Incorporating spatial variation of $c^2$ into WADG simply involves modifying the update kernel to compute 
\[
\bm{P}_q {\rm diag}\LRp{\bm{c}^2_q / {\bm{J}^k_q}}\bm{V}_q \widehat{\bm{M}}^{-1}\bm{A}^k_h\bm{Q},
\]
where $\bm{c}^2_q$ contains the evaluation of $c^2$ at quadrature points.  

Figure~\ref{fig:het} shows an example of a Gaussian pulse (placed away from the origin) propagating through a spherical domain with radially varying wavespeed.  The simulation was performed using the strong-weak DG formulation on 83762 elements of degree $N=4$.  For computational efficiency, different volume and surface kernels were used depending whether an element was planar or curved.  The same WADG update kernel was used for all elements in order to simultaneously accommodate spatial variation of $c^2$ everywhere in the domain and variation of $J$ over curvilinear elements.  
  \begin{figure}
  \centering
  \subfloat[Wavespeed $c^2$]{\label{subfig:c2}\includegraphics[width=.29\textwidth]{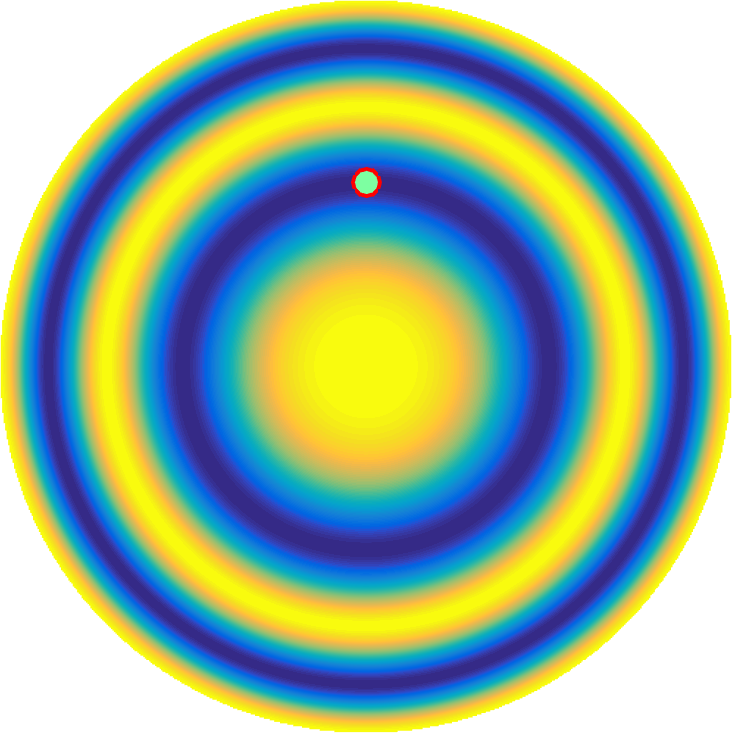}}
  \subfloat[Pressure at $t\approx .27$]{\includegraphics[width=.35\textwidth]{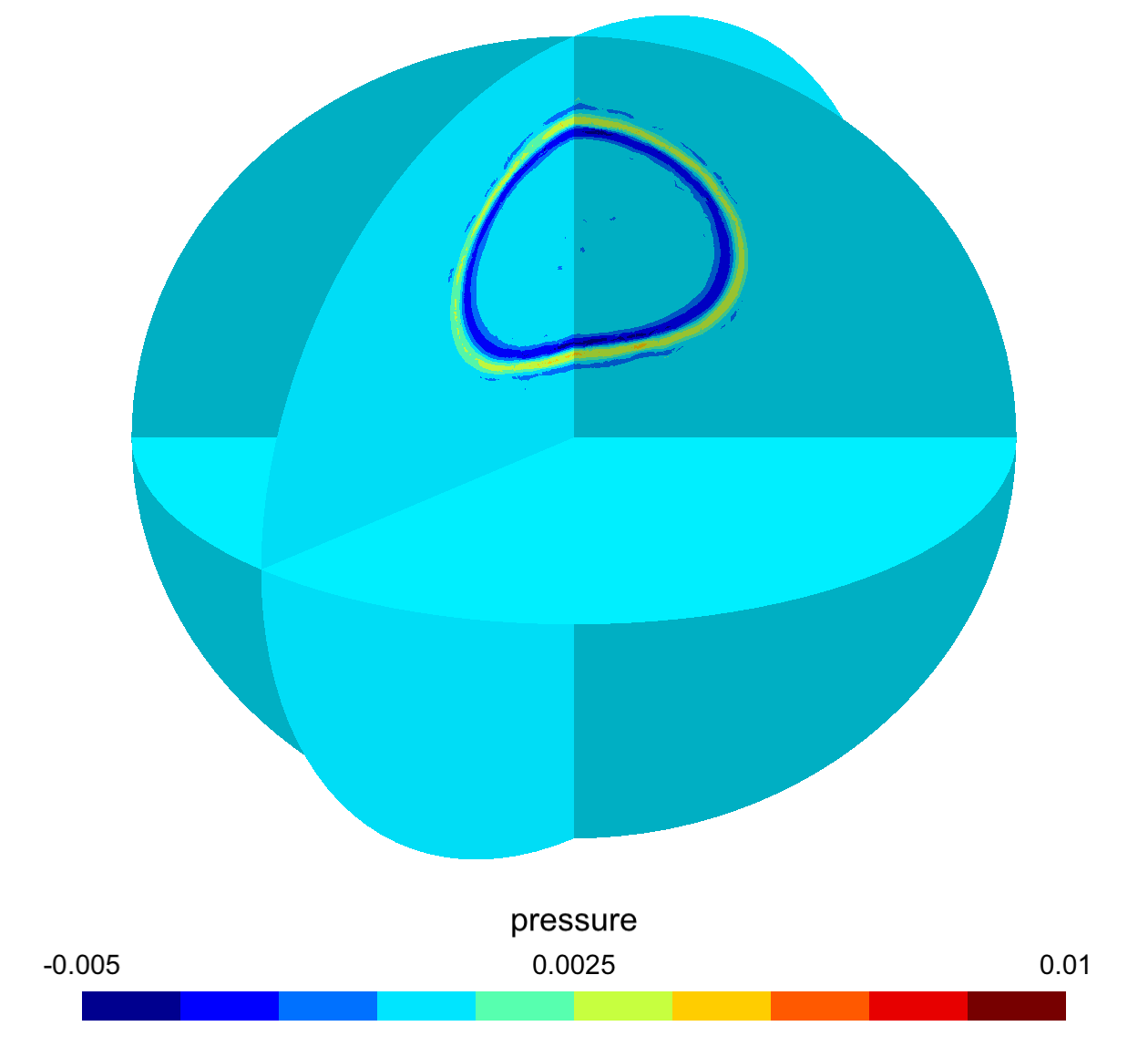}}
%  \hspace{-1em}
  \subfloat[Pressure at $t\approx .45$]{\includegraphics[width=.35\textwidth]{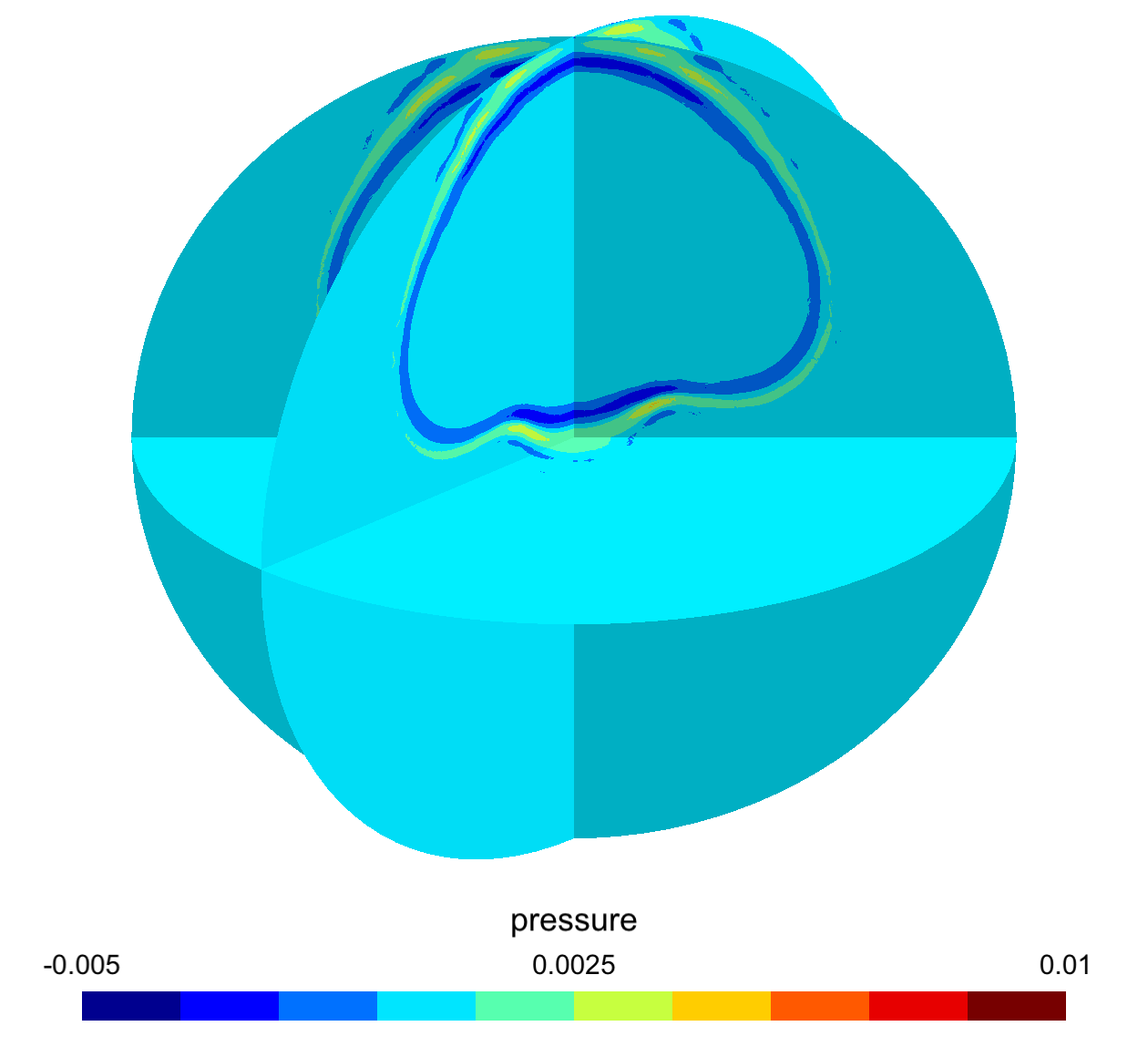}}
%  \hspace{-1em}
%  \subfloat[Pressure at $t\approx .64$]{\includegraphics[width=.33\textwidth]{figs/p7.png}}
  \caption{Acoustic wave propagation in heterogeneous media with curvilinear elements of degree $N=4$.  Wavespeed varies in the radial coordinate.  Initial pressure is taken to be a Gaussian pulse centered around the green circle (\ref{subfig:c2}).  }  
  \label{fig:het}
  \end{figure}

\section{Conclusions and future work}

This work introduces a weight-adjusted discontinuous Galerkin (WADG) method for curvilinear meshes based on a weight-adjusted approximation of the inverse mass matrix.  This approximation results in low asymptotic storage cost, and $L^2$ error estimates with explicit constants depending on the Jacobian $J$ provide sufficient conditions under which the WADG pseudo-projection behaves similarly to the $L^2$ projection.  Numerical comparisons indicate that WADG is also more robust with respect to curved mesh deformations than existing low storage DG methods in the literature \cite{warburton2013low}.  Computational results are also given which verify the high order convergence of WADG for solutions on curved domains, and performance analysis results on an Nvidia GTX 980 GPU are presented.  

Currently, a limitation of WADG methods is that the error bound depends on the regularity of the weighting function.  This excludes non-affine elements with singular mappings, such as transitional pyramid elements \cite{bergot2010higher, chan2015orthogonal}.  It was observed in \cite{bergot2013higher, chan2015orthogonal} that the $L^2$ projection errors for LSC-DG on pyramids stalls under both $h$ and $p$-refinement, and numerical experiments indicate that $L^2$ errors for non-affine pyramids using WADG pseudo-projection behave similarly.  These issues will be addressed in future work.  

\section{Acknowledgments}

The authors wish to thank \textsc{Total} US E\&P for permission to publish.  The authors also acknowledge the use of \textsc{Gmsh} \cite{geuzaine2009gmsh} for both mesh generation and visualization of high order finite element solutions \cite{remacle2007efficient}.  JC and TW are funded in part by a grant from \textsc{Total} E\&P Research and Technology USA.  

\bibliographystyle{unsrt}
\bibliography{curved}{}

\end{document}